\documentclass[a4paper,12pt]{article}
\usepackage[english]{babel}
\usepackage{hyperref}
\usepackage{amssymb,amsmath,amsthm}
\usepackage{enumerate,enumitem}
\usepackage{graphicx,color}
\usepackage{float}
\usepackage{epsfig}
\usepackage{tikz}
\usetikzlibrary{calc,decorations.pathmorphing,decorations.text,decorations.markings,matrix}
\usepackage{caption,subcaption}
\usepackage{refcount}
\usepackage[hmargin=2.5cm,vmargin=2.5cm]{geometry}
\usepackage{thmtools, mathtools, braket}
\usepackage[normalem]{ulem}
\usepackage{authblk}
\usepackage{xcolor} 
\usepackage{natbib}
\usepackage[nameinlink]{cleveref}
\usepackage{booktabs}
\usepackage{array}

\theoremstyle{plain}
\newtheorem{thm}{Thm}[section]

\newtheorem{claim}{Claim}[section]
\newtheorem{theorem}[thm]{Theorem}
\newtheorem{lemma}[thm]{Lemma}
\newtheorem{corollary}[thm]{Corollary}

\newtheorem{conjecture}[thm]{Conjecture}

\newenvironment{proof*}{\noindent \emph{Proof of the claim.}}{\hfill$\Diamond$}

\begin{document}

\title{Positive and negative $3$-energies of graphs}
    \author[1]{Zhengbo Chen}
    \author[2]{Zhouningxin Wang}
    \author[3]{Xiao-Dong Zhang}

	\affil[1]{\small School of Mathematical Sciences, Shanghai Jiao Tong University, Shanghai 200240, China}
    \affil[2]{\small School of Mathematical Sciences and LPMC, Nankai University, Tianjin 300071, China}
    \affil[3]{\small School of Mathematical Sciences, MOE-LSC, SHL-MAC, Shanghai Jiao Tong University, Shanghai 200240, China
	\linebreak Email:\{czb911, xiaodong\}@sjtu.edu.cn;
    wangzhou@nankai.edu.cn.}

\date{ }	
\maketitle

\begin{abstract}
For a simple $n$-vertex graph $G$, let $A_G$ be the adjacency matrix of $G$ and let $\lambda_{1}(G) \geq \lambda_{2}(G)\geq  \cdots \geq \lambda_{n}(G)$ denote the eigenvalues of $A_G$. For an integer $p\geq 2$, the positive $p$-energy and negative $p$-energy of a graph $G$, denoted $\mathcal{E}^+_p(G)$ and $\mathcal{E}^+_p(G)$, are defined as follows: $\mathcal{E}^+_p(G)=\sum_{\lambda_{i}(G)>0} |\lambda_{i}(G)|^{p}$ and $\mathcal{E}^-_p(G)=\sum_{\lambda_{i}(G)<0}|\lambda_{i}(G)|^{p}$. Q. Tang, Y. Liu, and W. Wang, and S. Akbari, H. Kumar, and B. Mohar, two groups conjectured that for any integer $p\geq 2$, every connected $n$-vertex graph $G$ satisfies that $\mathcal{E}^+_p({G}) \geq \mathcal{E}^+_p({P_{n})}$ and $\mathcal{E}^-_p({G}) \geq \mathcal{E}^-_p({K_{n})}$. The negative $p$-energy conjecture is so-far known to be true for $p \geq 4$ [Linear Algebra Appl., 724:96–107, 2025]. Improving upon this result, in this paper we establish the following two main results: 
(1) For any integer $p\geq 3$, every connected $n$-vertex graph $G$ satisfies $\mathcal{E}^-_p({G}) \geq \mathcal{E}^-_p({K_{n})}$. 
(2) Every connected $n$-vertex graph, other than $K_1$, $K_2$, and $P_3$, satisfies $\mathcal{E}_3^{+}(G)\geq\frac{\sqrt{5}}{2}n$.
\end{abstract}

\maketitle

\section{Introduction}
Let $A_G$ be the adjacency matrix of a simple graph $G$. Let $\lambda_{1}(G), \lambda_{2}(G), \ldots, \lambda_{n}(G)$ be the eigenvalues of $A_G$ such that $\lambda_{1}(G) \geq \lambda_{2}(G)\geq  \cdots \geq \lambda_{n}(G)$. 
For an integer $p\geq 2$, the \emph{positive $p$-energy} and \emph{negative $p$-energy} of a graph $G$, denoted $\mathcal{E}^+_p(G)$ and $\mathcal{E}^+_p(G)$, are defined as follows: 
$$\mathcal{E}^+_p(G)=\sum_{\lambda_{i}(G)>0} |\lambda_{i}(G)|^{p},~~~\text{and}~~~
\mathcal{E}^-_p(G)=\sum_{\lambda_{i}(G)<0}|\lambda_{i}(G)|^{p}.$$

Positive and negative $p$-energies are useful parameters related to other graph characteristics, such as chromatic number and fractional chromatic number; We refer interested readers to a recent work~\cite{ELPHICK2026104252}.

\begin{conjecture}[\cite{elphick2016conjectured}]\label{conj p=2}
For every connected $n$-vertex graph $G$, $$\min\{\mathcal{E}_2^+({G}), ~\mathcal{E}_2^-({G})\} \geq n - 1.$$
\end{conjecture}

\Cref{conj p=2} has been verified for many classes of graphs: bipartite graphs, regular graphs, complete $q$-partite graphs, hyper-energetic graphs, and barbell graphs in \cite{elphick2016conjectured}, and extended barbell graphs, unicyclic graphs, graphs with two positive eigenvalues, graphs with a certain fraction of positive or negative eigenvalues, cactus graphs in \cite{Abiad2023}, and $d$-regular disconnected graphs with no clique component in \cite{Elphick2024}. S. Zhang~\cite{Zhang2024} proved that 
$\min\{\mathcal{E}^+_2(G),\, \mathcal{E}^-_2(G)\} \ge n-\gamma,$ where $\gamma$ is the domination number of $G$. In the same paper, Zhang also showed that $
\mathcal{E}^{+}_2(G) \ge m^{6/7 - o(1)}$ and $\mathcal{E}^{-}_2(G) = \Omega(m^{1/2})$ where $m=|E(G)|$ with exponents in both bounds being optimal. The best-known bound towards ~\Cref{conj p=2} was obtained by S. Akbari, H. Kumar, and B. Mohar \cite{akbari2025linear} in 2025, who utilized a super-additivity property of square energies to prove that $$\min\{\mathcal{E}^{+}_2(G), \, \mathcal{E}^{-}_2(G)\} \ge \frac{3n}{4}.$$ Nevertheless, \Cref{conj p=2} remains open and continues to stimulate active research interest.

Let $P_n$ denote the path on $n$ vertices,  $K_n$ denote the complete graph on $n$ vertices, and $K_{1,n}$ denote a star graph with $n$ vertices of degree $1$. 
Note that for each integer $p\geq 2$, $\mathcal{E}_p^-(K_n)=n-1$ and $\mathcal{E}_p^+(P_n)=\sum\limits_{k=1}^{\lfloor\frac{n}{2}\rfloor} 2^p\cos^{p}\!\Big(\frac{k\pi}{n+1}\Big)$; In particular, $\mathcal{E}_2^+(P_n)=n-1$. 
Following the same line of study, the next two conjectures can be seen as natural generalizations of \Cref{conj p=2}.

\begin{conjecture}[\cite{TLW2025}]\label{conj:positive}
Let $n$ and $p$ be two positive integers with $p\geq 2$. For every connected $n$-vertex graph $G$, $$\mathcal{E}^+_p({G}) \geq \mathcal{E}^+_p({P_{n})}.$$
\end{conjecture}

\begin{conjecture}[\cite{AKBARI202596}]\label{conj:negative}
Let $n$ and $p$ be two positive integers with $p\geq 2$. For every connected $n$-vertex graph $G$, $$\mathcal{E}^-_p({G}) \geq \mathcal{E}^-_p({K_{n})}.$$
\end{conjecture}

For connected graphs $G$ and all $p\ge 2$, S. Akbari, H. Kumar, B. Mohar, and S. Pragada~\cite{AKBARI202596} obtained upper bounds on the positive and negative $p$-energies: $\mathcal{E}_p^{+}(G)\le (n-1)^{p}$ and $\mathcal{E}_p^{-}(G)\le (n/2)^{p}$, showing that neither part can grow faster than a fixed polynomial in $n$. In a complementary direction, C. Elphick, Q. Tang, and S. Zhang~\cite{ELPHICK2026104252} conjectured that for every connected $n$-vertex graphs $G$ and every $p$ with $0\le p\le 2$, 
$\min\{\mathcal{E}_p^{+}(G),\, \mathcal{E}_p^{-}(G)\}\ge (n-1)^{p/2},$
and they verified this bound for $p$ with $0\leq p\leq1$. Another conjecture concerning $\mathcal{E}_p(G)=\mathcal{E}_p^{+}(G)+ \mathcal{E}_p^{-}(G)$ proposed by V. Nikiforov~\cite{NIKIFOROV201682} is that, for every tree $T$, $\mathcal{E}_p(K_{1,{n-1}})\le \mathcal{E}_p(T)\le \mathcal{E}_p(P_n)$ holds for $1\le p\le 2$, while $\mathcal{E}_p(P_n)\le \mathcal{E}_p(T)\le \mathcal{E}_p(K_{1,n-1})$ holds for $p>2$; for some related results see ~\cite{Arizmendi2023TheGE,Arizmendi2023}.

In a very recent work \cite{AKBARI202596} by S. Akbari, H. Kumar, B. Mohar, and S. Pragada, \Cref{conj:negative} has been verified for $p \geq 4$. 

\begin{theorem}[\cite{AKBARI202596}]
Let $n$ and $p$ be two positive integers with $p\geq 4$. For every connected $n$-vertex graph $G$, $\mathcal{E}^-_p({G}) \geq \mathcal{E}^-_p({K_{n})}.$
\end{theorem}

In this paper, we take a step towards resolving \Cref{conj:negative} by showing that it holds for $p\geq 3$. 

\begin{theorem}\label{thm:3-energy-main}
Let $n$ be a positive integer. If $G$ is a connected $n$-vertex graph, not isomorphic to $K_n$ or $P_3$, then $\mathcal{E}_3^{-}(G) \geq n.$
\end{theorem}

Note that $\mathcal{E}_3^{-}(P_3)=2\sqrt{2}>2=\mathcal{E}_3^{-}(K_3)$, and recall that $\mathcal{E}_p^-(K_n)=n-1$ for every integer $p\geq 2$. Our result thus implies that \Cref{conj:negative} is true for $p=3$.

We further remark that proving \Cref{conj:negative} for a smaller value of $p$ is more challenging than for a larger one. Note that \Cref{conj:positive} does not exhibit this phenomenon. For any real number $r\geq 2$, let $\mathcal{E}^-_r(G)=\sum\limits_{\lambda_{i}(G)<0}|\lambda_{i}(G)|^{r}$.

\begin{lemma}\label{lem:r to r'}
Let $r$ and $r'$ be two real numbers such that $2\leq r\leq r'$.  Assume that $G$ is a connected graph. 
If $\mathcal{E}_r^-(G)\geq\mathcal{E}_r^-(K_n)$, then $\mathcal{E}_{r'}^-(G)\geq\mathcal{E}_{r'}^-(K_n)$.
\end{lemma}

Let $\theta$ denote the number of negative eigenvalues of $G$. Since $\lambda_1 >0$, we have $\theta \leq n-1$. By Hölder's inequality, for $r \leq r'$, we have the following inequality for $\mathcal{E}_r^-(G)$ and $\mathcal{E}_{r'}^-(G)$: $$(\mathcal{E}_{r'}^-(G))^{\frac{r}{r'}} \, \theta^{\frac{r'-r}{r'}} \;\geq\; \mathcal{E}_r^-(G) \;\geq\; \mathcal{E}_r^-(K_n) = n-1.$$ Therefore, we obtain $\mathcal{E}_{r'}^-(G) \geq (\frac{(n-1)^{r'}}{\theta^{r'-r}})^{\frac{1}{r}}.$ Since $\theta \leq n-1$, it follows that $\mathcal{E}_{r'}^-(G) \geq n-1.$ Thus, we have shown that $\mathcal{E}_{r'}^-(G) \geq \mathcal{E}_{r'}^-(K_n)$, completing the proof of the lemma.

\medskip
In particular, \Cref{lem:r to r'} holds for $r=p$ where $p$ is a positive integer. Combining all discussion above, we have the following corollary.

\begin{corollary}
Let $n$ and $p$ be two positive integers with $p\geq 3$. For every connected $n$-vertex graph $G$, $\mathcal{E}^-_p({G}) \geq \mathcal{E}^-_p({K_{n})}.$
\end{corollary}

Also we note that the $p=2$ case of \Cref{conj:negative} is the most changeling one. Together with \Cref{lem:r to r'}, our result of \Cref{thm:3-energy-main} suggests a natural direction for future work: investigating the real-valued graph energy, in particular, the $r$-energy of $G$ for rational $r \in (2, 3)$. 

Compared to the previous work \cite{AKBARI202596}, which established the result for $p \geq 4$, our main effort is devoted to handling a somewhat ``bad'' graph: the path $P_3$. This is due to the fact that the inequality $2<\mathcal{E}_3^-(P_3)<3$ holds, which introduces complications in the inductive argument. In this work, when analyzing the negative $3$-energy, we carefully examine the structure possibly containing an induced $P_3$ and employ some technical trick to proceed further. We note that this approach is effective here because the only obstruction is $P_3$, aside from complete graphs $K_n$'s. In contrast, the case $p=2$ is considerably more complicated, as all paths $P_n$ constitute challenging base cases for induction.

\medskip
Besides, regarding~\Cref{conj:positive}, we establish the following main result concerning the positive $3$-energy of connected graphs. 

\begin{theorem}\label{thm:positive-main}
Let $n$ be a positive integer. If $G$ is a connected $n$-vertex graph, not isomorphic to $K_1,K_2$, and $P_3$, then $\mathcal{E}_3^{+}(G)  \geq \frac{\sqrt{5}}{2}n$.
\end{theorem}

In the end, we state some useful results. 

\begin{theorem}{\rm \cite{AKBARI202596}}\label{thm: additivity}
    Let $A = [A_{ij}]_{i,j=1}^k$ be a Hermitian matrix partitioned into $k^2$ blocks, where $A_{ii}$ are square matrices. For any real number $r \geq 1$, the following hold: $$\mathcal{E}_r^+(A) \geq \sum_{i=1}^k \mathcal{E}_r^+(A_{ii}) \quad\text{~and~}  \quad\mathcal{E}_r^-(A) \geq \sum_{i=1}^k \mathcal{E}_r^-(A_{ii}).$$
\end{theorem}

\begin{theorem}[Eigenvalue Interlacing Theorem~{\cite{CRS2010}}]\label{thm:alternate} 
For any $n$-vertex graph $G$ and a subset $A\subset V(G)$, for $i \in \{1, 2, \ldots, n-|A|\}$, we have $$\lambda_i(G)\geq \lambda_i(G-A)\geq\lambda_{i+|A|}(G).$$
\end{theorem}

\begin{lemma}[\cite{CRS2010}]\label{lem:complete graph+star+cycle} 
The following claims hold.
\begin{enumerate}[label=(\arabic*)]
\setlength{\itemsep}{0em}
    \item\label{Kn} $\operatorname{Spec}(K_n)=\{n-1,(-1)^{n-1}\}$, $\mathcal{E}_3^-(K_n)=n-1$, and $ \mathcal{E}_3^+(K_n)=(n-1)^3$.
    \item\label{K1n} $\operatorname{Spec}(K_{1,n})=\{-\sqrt{n},\, 0^{n-1},\, \sqrt{n}\}$ and  $\mathcal{E}_3^-(K_{1,n})=\mathcal{E}_3^+(K_{1,n})=n\sqrt{n}$.
\end{enumerate}
\end{lemma}

The remainder of the paper is organized as follows: In the next section, we prove \Cref{conj:positive} for $p=3$ and graphs with $n$ vertices and $1.2n$ edges, and then present the proof of \Cref{thm:positive-main}. In \Cref{sec:negative-energy}, we first characterize some special graphs that satisfy \Cref{conj:negative} for $p=3$ and serve as the base case for the induction. Finally, in \Cref{sec:mainproof}, we provide the proof of \Cref{thm:3-energy-main}.

\section[positive]{On the positive $3$-energy of a graph}

We use a program to enumerate and compute all connected $n$-vertex graphs $G$ with $n \leq 10$; the algorithm used is provided in 
\href{https://github.com/chenzb911/graph-energy-verification}{https://github.com/chenzb911/graph-energy-verification}.

\begin{lemma}\label{lem:PositiveUpToTenVertices}
Let $n$ be an integer such that $n\leq 10$. Every connected $n$-vertex graph $G$ satisfies that $\mathcal{E}_3^+(G)\geq \mathcal{E}_3^+(P_n)$. Moreover, if $G$ is not isomorphic to any of $K_1,K_2$, and $P_3$, then $\mathcal{E}_3^{+}(G)  \geq \frac{\sqrt{5}}{2}n$.
\end{lemma}

Remark that $\mathcal{E}_3^{+}(P_4) = 2\sqrt{5} = \frac{\sqrt{5}}{2} \times |P_4|$, suggesting that $\frac{\sqrt{5}}{2} n$ is the tightest bound we can currently achieve (if we do not exclude $P_4$).

\medskip
\noindent
{\bf \Cref*{thm:positive-main}.} \emph{Let $n$ be a positive integer. If $G$ is a connected $n$-vertex graph, not isomorphic to $K_1,K_2$, and $P_3$, then $\mathcal{E}_3^{+}(G)  \geq \frac{\sqrt{5}}{2}n$.}

\begin{proof}
Let $n$ be a positive integer. Let $G$ be a connected $n$-vertex graph, and we proceed by induction on $n$. The assertion has been verified using a computer for all graphs up to $10$ vertices as shown in~\Cref{lem:PositiveUpToTenVertices}. From now on, we assume that $n\geq 11$.

Let $\Delta$ be the maximum degree of $G$ and 
let $w$ be a vertex with $d_G(w)=\Delta$. Assume that $T$ is a spanning tree of $G$ rooted at the vertex $w$. If there exists an edge of $E(T)$ incident to the vertex $w$, say $wu$, such that each of the two connected components $T_{a}$ and $T_{b}$ of $T-wu$ has at least $4$ vertices, then neither $G[V(T_{a})]$ nor $G[V(T_{b})]$ is isomorphic to any of $\{K_1, K_2, P_3\}$. By~\Cref{thm: additivity} and the induction hypothesis, $\mathcal{E}_3^{+}(G)\geq \mathcal{E}_3^{+}(G[V(T_{1})])+ \mathcal{E}_3^{+}(G[V(T_{2})]) \geq \frac{\sqrt{5}}{2}n$. Therefore, for each edge $wu$ of $E(T)$ incident to $w$, one of the two connected components of $T-wu$ has at most $3$ vertices.

We next consider the subgraph $G-w$. Let $T_1, \ldots, T_s$ denote the connected trees in $T-w$. 
We first claim that $G-w$ contains no subgraph isomorphic to $P_4$. Assume, to the contrary, that such a $4$-path $Q$ exists, and we choose one that intersects the minimum number of $T_i$'s. If there are $i,j\in [s]$ such that $V(Q)\subset V(T_i)\cup V(T_j)$, then as $|Q|=4$, $|T_i|+|T_j|\leq \max\{3+1, 3+2, 3+3, 2+2\}= 6$; Otherwise, there are either $V(Q)\subset V(T_i)\cup V(T_j)\cup V(T_k)$ with $|T_i|=2, |T_j|=1,$ and $|T_k|=1$, or $V(Q)\subset V(T_i)\cup V(T_j)\cup V(T_k)\cup V(T_\ell)$ with $|T_i|=|T_j|=|T_k|=|T_\ell|=1$, in both cases, $\sum |T_d|\leq \max\{2+1+1, 1+1+1+1\}=4$. Let $T^*$ denote the union of trees $T_i$'s which intersect the chosen $Q$. Since $Q\subset T^*$ and by the analysis above, we know that $4\leq |T^*|\leq 6$. Furthermore, $|T-T^*|\geq 11-6\geq 5$. By~\Cref{thm: additivity} and the induction hypothesis,
$\mathcal{E}^+_3(G)\geq\mathcal{E}^+_3(G[V(T^*)])+\mathcal{E}^+_3(G[V(T-T^*)]) \geq |V(T^*)|+|V(G)-V(T^*)|=\frac{\sqrt{5}}{2}n.$

Therefore, $G-w$ is a disjoint
union of $K_1, K_2, P_3,$ and $K_3$. Let $n_i$ denote the number of $K_i$'s in $G-w$ for $i\in [3]$, and let $n_{3^-}$ denote the number of $P_3$'s in $G-w$. Let $$s:= n_1+n_2+n_3+n_{3^-}.$$ Note that $s\leq d_G(w)\leq n-1$ and $n=n_1+2n_2+3n_3+3n_{3^-}+1$. Note that as $n\geq 11$, we have that $s\geq 4$.
Since $G$ contains $K_{1,s}$ as a subgraph, $\lambda_1(G)\geq\lambda_1(K_{1,s})=\sqrt{s}$. We consider two possibilities based on the value of $n_3$.

\medskip
\noindent
{\bf Case 1}. \emph{Assume that $n_3=0$.} 
\medskip

In this case, since $G-w$ contains no $K_3$, $\lambda_1(G-w)\leq \max\{\lambda_1(K_1),\lambda_1(K_2),\lambda_1(P_3)\}=\sqrt{2}$. 
By~\Cref{thm:alternate}, we have that
\[
\begin{aligned}
\mathcal{E}^+_3(G) 
&\geq \mathcal{E}^+_3(G-w)+\lambda_1(G)^3-\lambda_1(G-w)^3\\
&=\big(n_2\mathcal{E}^+_3(K_2)+n_{3^-}\mathcal{E}^+_3(P_3))+\lambda_1(G)^3-\lambda_1(G-w)^3\\
&\geq n_2+2\sqrt{2}n_{3^-}+s\sqrt{s}-2\sqrt{2}.
\end{aligned}
\]
Let $f_s(n_1,n_2,n_{3^-}):=n_2+2\sqrt{2}n_{3^-}+s\sqrt{s}-2\sqrt{2}- \frac{\sqrt{5}}{2}n$. It suffices to show that $f_s(n_1,n_2,n_{3^-})\geq 0$.
Substituting $n_1 = s - n_2 - n_{3^-}$ into $n$ gives
$n = s + n_2 + 2 n_{3^-} + 1.$ Thus $$f_s(n_1,n_2,n_{3^-})=(\sqrt{s}- \frac{\sqrt{5}}{2})s+(1-\frac{\sqrt{5}}{2})n_2+(2\sqrt{2}-\sqrt{5})n_{3^-}-2\sqrt{2}- \frac{\sqrt{5}}{2}.$$ Note that $1 - \frac{\sqrt{5}}{2} < 0,$ and $2\sqrt{2} -\sqrt{5} > 0.$ Hence, for a fixed $s = n_1 + n_2 + n_{3^-}$, a decrease in $n_{3^-}$ leads to a corresponding increase in $n_1$ or $n_2$, and consequently, a decrease in $f_s(n_1, n_2, n_{3^-})$; and also a decrease in $n_1$ leads to a corresponding increase in $n_2$, and consequently, again a decrease in $f_s(n_1, n_2, n_{3^-})$.  
Under the constraints of
$\min\{n_1, n_2, n_{3^-}\}\geq 0$ and $n_1 + n_2 + n_{3^-} = s$, we have that $$f_s(n_1, n_2, n_{3^-})\geq f_s(n_1+n_{3^-}, n_2,0 )\geq f_s(0, n_1+n_2+n_{3^-},0)=f_s(0,s,0),$$
where
$$
f_s(0,s,0)=s\bigl(\sqrt{s} + 1 - \sqrt{5}\bigr) - 2\sqrt{2} - \frac{\sqrt{5}}{2}.
$$

Note that $f_s(0,s,0)\geq 0$ when $s\geq 5$. Recall that $n\geq 11$, $s\geq 4$, and $n=n_1+2n_2+3n_{3^-}+1$. It remains to consider the case when $s=4$. We only have the following four possibilities: $(n_1,n_2,n_{3^-})\in \{(0,0,4),~(0,1,3),~(0,2,2),~(1,0,3)\}$. Note that $$f_4(n_1,n_2,n_{3^-})=8-2\sqrt{2}-\frac{5\sqrt{5}}{2}+(1-\frac{\sqrt{5}}{2})n_2+(2\sqrt{2}-\sqrt{5})n_{3^-}.$$ By the monotonicity of $f_4(n_1,n_2,n_{3^-})$ and some numerical computation, we obtain the following result: $f_4(0,1,3)\geq f_4(0,2,2)= 10 + 2\sqrt{2} - \frac{11\sqrt{5}}{2} \approx 0.53,~\text{and} ~f_4(0,0,4)\geq f_4(1,0,3) = 8 + 4\sqrt{2} - \frac{11\sqrt{5}}{2} \approx 1.35.$ Therefore, when $n_3=0$, $\mathcal{E}^+_3(G) \geq \frac{\sqrt{5}}{2}n.$

\medskip
\noindent
{\bf Case 2}. \emph{Assume that $n_3\geq 1$.}
\medskip

In this case, $\lambda_1(G-w)\leq \max\{\lambda_1(K_1),\lambda_1(K_2),\lambda_1(P_3),\lambda_1(K_3)\}=2$. 
By~\Cref{thm:alternate}, we have that
\[
\begin{aligned}
\mathcal{E}^+_3(G) 
&\ge \mathcal{E}^+_3(G-w) + \lambda_1(G)^3 - \lambda_1(G-w)^3 \\
&=\big(n_2 \mathcal{E}^+_3(K_2)+n_{3^-} \mathcal{E}^+_3(P_3)+n_3 \mathcal{E}^+_3(K_3) \big)+\lambda_1(G)^3-\lambda_1(G-w)^3 \\
&\ge n_2+2\sqrt{2}n_{3^-}+8n_3+s\sqrt{s}-8.
\end{aligned}
\]
Recall that $s = n_1 + n_2 + n_{3^-} + n_3$. Let
$g_s(n_1,n_2,n_{3^-},n_3):= s\sqrt{s} + n_2 + 2\sqrt{2} n_{3^-} + 8 n_3 - 8 - \frac{\sqrt{5}}{2} n$. Thus it suffices to prove that $g_s(n_1,n_2,n_{3^-},n_3)\geq 0$. Substituting $n_1 = s - n_2 - n_{3^-} - n_3$ into $n$ gives $n = s + n_2 + 2 n_{3^-} + 2 n_3 + 1.$ Hence, we have that $$g_s(n_1,n_2,n_{3^-},n_3)=(\sqrt{s} - \frac{\sqrt{5}}{2})s + (1- \frac{\sqrt{5}}{2}) n_2 + (2\sqrt{2}-\sqrt{5}) n_{3^-} + \left(8 - \sqrt{5}\right) n_3 - 8 - \frac{\sqrt{5}}{2}.$$ Under the constraints of $n_1,n_2,n_{3^-} \ge 0$, $n_3\geq 1$, and $n_1+n_2+n_{3^-}+n_3=s$, we have that
$$
\begin{aligned}
g_s(n_1, n_2, n_{3^-},n_3)
&\geq g_s(n_1+n_3-1, n_2,n_{3^-},1 )\\
&\geq g_s(n_1+n_3-1+n_{3^-}, n_2,0,1 )\\
&\geq g_s(0, n_1+n_2+n_3+n_{3^-}-1,0,1 )\\
&= g_s(0, s-1, 0, 1),
\end{aligned}
$$
where
$$
g_s(0, s-1, 0, 1) =(\sqrt{s}+1-\sqrt{5})s - 1 - \sqrt{5}.
$$ Note that the first inequality follows from the fact that, when fixing $s$, $n_2$, and $n_{3^-}$, a decrease in $n_3$ leads to a corresponding increase in $n_1$, and consequently, a decrease in $g_s$; the second inequality follows from the fact that, when fixing $s$ and $n_2$, a decrease in $n_{3^-}$ leads to a corresponding increase in $n_1$, and consequently, a decrease in $g_s$; the third inequality follows from the fact that, when fixing $s$, a decrease in $n_1$ leads to a corresponding increase in $n_2$, and consequently, a decrease in $g_s$.

It is easy to see that $g_s(0, s-1, 0, 1) \ge 0$ when $s \ge 5$. It remains to consider the case when $s=4$. Recall that in this case $n_3\geq 1$ and $s = n_1 + n_2 + n_{3^-} + n_3$. We have the following possibilities: $(n_1,n_2,n_{3^-},n_3) \in \{ (0,0,0,4),~(0,0,1,3),~(0,0,2,2),~(0,0,3,1),~(0,1,0,3),$\\ $(0,1,1,2),~
(0,1,2,1),~(0,2,0,2),~(0,2,1,1),~(1,0,0,3),~(1,0,1,2),~(1,0,2,1)\}$. Note that $$g_4(n_1,n_2,n_{3^-},n_3)=\left(1 - \frac{\sqrt{5}}{2}\right) n_2 + \left(2\sqrt{2} - \sqrt{5}\right) n_{3^-} + \left(8 - \sqrt{5}\right) n_3 - \frac{5\sqrt{5}}{2}.$$ By the monotonicity of $g_4(n_1,n_2,n_{3^-},n_3)$ and some numerical computation, we have that $$ 
\begin{aligned}
&g_4(0,0,1,3)\geq g_4(0,0,2,2)\geq g_4(0,0,3,1)=8 + 6\sqrt{2} - \frac{13\sqrt{5}}{2}\approx 1.95,\\
&g_4(0,1,0,3)\geq g_4(0,1,1,2)\geq g_4(0,1,2,1)=9 + 4\sqrt{2} - 6\sqrt{5}\approx 1.24,\\
&g_4(0,0,0,4)\geq g_4(0,2,0,2)\geq g_4(0,2,1,1)=10+2\sqrt{2} - \frac{11\sqrt{5}}{2}\approx0.53,\\
&g_4(1,0,0,3)\geq g_4(1,0,1,2)\geq g_4(1,0,2,1)=8 + 4\sqrt{2} - \frac{11\sqrt{5}}{2}\approx1.36.
\end{aligned}$$

Hence, when $n_3\geq 1$, $\mathcal{E}^+_3(G) \geq \frac{\sqrt{5}}{2} n$.
\end{proof}

\section[negative]{On the negative $3$-energy of a graph}\label{sec:negative-energy}

We employ a program to enumerate and compute all connected $n$-vertex graphs $G$ with $n \leq 10$, see \href{https://github.com/chenzb911/graph-energy-verification}{https://github.com/chenzb911/graph-energy-verification}.

\begin{lemma}\label{lem:UpToTenV}
Let $n$ be an integer such that $n\leq 10$. Every connected $n$-vertex graph $G$ satisfies $\mathcal{E}_3^-(G)\geq \mathcal{E}_3^-(K_n)$. Furthermore, if $G$ is isomorphic to neither $P_3$ nor $K_n$, then $\mathcal{E}_3^-(G)\geq n$.
\end{lemma}

We actually compute the spectrum of some small graphs on at most $6$ vertices using algorithms, and list the results in the table below. \Cref{table:3neg-energy} will be frequently referenced throughout the analysis of small configurations in the proof.

\begin{table}[htbp]
\centering
\renewcommand{\arraystretch}{2}
\setlength{\tabcolsep}{7pt}
\small
\begin{tabular}{|c|>{\centering\arraybackslash}m{2.8cm}|>{\centering\arraybackslash}m{8.5cm}|>{\centering\arraybackslash}m{1.5cm}|}
\hline
\textbf{Graph} & \textbf{Figure} & \textbf{Spectrum $\{\lambda_i\}$} & \textbf{$\mathcal{E}_3^{-}(G)$} \\
\hline

$P_3$ &
\begin{minipage}{2.8cm}
\centering
\begin{tikzpicture}[scale=0.6, every node/.style={draw,circle,fill=black,inner sep=1pt}]
\node (1) at (0,0) {}; \node (2) at (1,0) {}; \node (3) at (2,0) {};
\draw (1)--(2)--(3);
\end{tikzpicture}
\end{minipage}
& $\{1.4142,\, 0,\,-1.4142\}$ & $2\sqrt{2}$ \\
\hline

$P_4$ &
\begin{minipage}{2.8cm}
\centering
\begin{tikzpicture}[scale=0.6, every node/.style={draw,circle,fill=black,inner sep=1pt}]
\node (1) at (0,0) {}; \node (2) at (1,0) {}; \node (3) at (2,0) {}; \node (4) at (3,0) {};
\draw (1)--(2)--(3)--(4);
\end{tikzpicture}
\end{minipage}
& $\{1.618,\, 0.618,\,-0.618,\,-1.618\}$ & $4.472$ \\
\hline

$P_5$&
\begin{minipage}{2.8cm}
\centering
\begin{tikzpicture}[scale=0.6, every node/.style={draw,circle,fill=black,inner sep=1pt}]
\node (1) at (0,0) {}; \node (2) at (1,0) {}; \node (3) at (2,0) {}; \node (4) at (3,0) {}; \node (5) at (4,0) {};
\draw (1)--(2)--(3)--(4)--(5);
\end{tikzpicture}
\end{minipage}
& $\{1.732,\, 1,\, 0,\,-1,\,-1.732\}$ & $6.196$ \\
\hline

$K_{1,3}$ &
\begin{minipage}{2.8cm}
\centering
\begin{tikzpicture}[scale=0.6, every node/.style={draw,circle,fill=black,inner sep=1pt}]
\node (1) at (0,0) {}; \node (2) at (1,0) {}; \node (3) at (2,0) {}; \node (4) at (1,1) {};
\draw (1)--(2)--(3); \draw (2)--(4);
\end{tikzpicture}
\end{minipage}
& $\{1.732,\, 0,\, 0,\,-1.732\}$ & $5.196$ \\
\hline

$C_4$ &
\begin{minipage}{2.8cm}
\centering
\begin{tikzpicture}[scale=0.6, every node/.style={draw,circle,fill=black,inner sep=1pt}]
\node (1) at (0,0) {}; \node (2) at (1,0) {}; \node (3) at (1,1) {}; \node (4) at (0,1) {};
\draw (1)--(2)--(3)--(4)--(1);
\end{tikzpicture}
\end{minipage}
& $\{2,\, 0,\, 0,\,-2\}$ & $8$ \\
\hline

$C_5$ &
\begin{minipage}{2.8cm}
\centering
\begin{tikzpicture}[scale=0.4, every node/.style={draw,circle,fill=black,inner sep=1pt}]
\node (1) at (90:1) {}; 
\node (2) at (162:1) {}; 
\node (3) at (234:1) {}; 
\node (4) at (306:1) {}; 
\node (5) at (18:1) {};
\draw (1)--(2)--(3)--(4)--(5)--(1);
\end{tikzpicture}
\end{minipage}
& $\{2,~0.618,~0.618,~-1.618,~-1.618\}$ & $8.472$ \\
\hline

$H_1$ &
\begin{minipage}{2.8cm}
\centering
\begin{tikzpicture}[scale=0.6, every node/.style={draw,circle,fill=black,inner sep=1pt}]
\node (1) at (0,0) {}; \node (2) at (1,0) {}; \node (3) at (0.5,0.866) {}; \node (4) at (2,0) {};
\draw (1)--(2)--(3)--(1); \draw (2)--(4);
\end{tikzpicture}
\end{minipage}
& $\{2.170,\, 0.311,\,-1,\,-1.481\}$ & $4.249$ \\
\hline

$H_2$ &
\begin{minipage}{2.8cm}
\centering
\begin{tikzpicture}[scale=0.55, every node/.style={draw,circle,fill=black,inner sep=1pt}]
\node (1) at (0,0) {}; \node (2) at (1,0) {}; \node (3) at (0.5,0.866) {}; \node (4) at (2,0) {}; \node (5) at (3,0) {};
\draw (1)--(2)--(3)--(1); \draw (2)--(4)--(5);
\end{tikzpicture}
\end{minipage}
& $\{2.214,\, 1,\,-0.539,\,-1,\,-1.675\}$ & $5.857$ \\
\hline

$H_3$ &
\begin{minipage}{2.8cm}
\centering
\begin{tikzpicture}[scale=0.6, every node/.style={draw,circle,fill=black,inner sep=1pt}]
\node (1) at (0,0) {}; \node (2) at (1,0) {}; \node (3) at (2,0) {}; \node (4) at (1,1) {}; \node (5) at (3,0) {};
\draw (1)--(2)--(3)--(5); \draw (2)--(4);
\end{tikzpicture}
\end{minipage}
& $\{1.848,\, 0.765,\, 0,\,-0.765,\,-1.848\}$ & $6.757$ \\
\hline

$H_4$ &
\begin{minipage}{2.8cm}
\centering
\begin{tikzpicture}[scale=0.6, every node/.style={draw,circle,fill=black,inner sep=1pt}]
\node (1) at (0,0) {}; \node (2) at (1,0) {}; \node (3) at (0.5,0.866) {}; \node (4) at (2,0) {}; \node (5) at (1.5,0.866) {};
\draw (1)--(2)--(3)--(1); \draw (2)--(4); \draw (3)--(5);
\end{tikzpicture}
\end{minipage}
& $\{2.303,\, 0.618,\, 0,\,-1.303,\,-1.618\}$ & $6.447$ \\
\hline

$H_5$ &
\begin{minipage}{2.8cm}
\centering
\begin{tikzpicture}[scale=0.6, every node/.style={draw,circle,fill=black,inner sep=1pt}]
\node (1) at (0,0) {}; \node (2) at (1,0) {}; \node (3) at (0.5,0.866) {}; \node (4) at (2,0) {}; \node (5) at (1.5,0.866) {}; \node (6) at (3,0) {};
\draw (1)--(2)--(3)--(1); \draw (2)--(4); \draw (2)--(5)--(4); \draw (4)--(6);
\end{tikzpicture}
\end{minipage}
& $\{2.629,\, 1.230,\, 0.140,\,-1,\,-1.320,\,-1.678\}$ & $8.026$ \\
\hline

$H_6$ &
\begin{minipage}{2.8cm}
\centering
\begin{tikzpicture}[scale=0.6, every node/.style={draw,circle,fill=black,inner sep=1pt}]
\node (1) at (0,0) {}; \node (2) at (1,0) {}; \node (3) at (1,1) {}; \node (4) at (0,1) {}; \node (5) at (2,1) {};
\draw (1)--(2)--(3)--(4)--(1); \draw (3)--(5);
\end{tikzpicture}
\end{minipage}
& $\{2.136,\, 0.662,\, 0,\,-0.662,\,-2.136\}$ & $10.032$ \\
\hline

$H_7$ &
\begin{minipage}{2.8cm}
\centering
\begin{tikzpicture}[scale=0.6, every node/.style={draw,circle,fill=black,inner sep=1pt}]
\node (1) at (0,0) {}; \node (2) at (1,0) {}; \node (3) at (1,1) {}; \node (4) at (0,1) {};
\node (5) at (2,1) {};
\draw (1)--(2)--(3)--(4)--(1);
\draw (3)--(5)--(2);
\end{tikzpicture}
\end{minipage}
& $\{2.481,\, 0.688,\, 0,\,-1.170,\,-2\}$ & $9.602$ \\
\hline

$H_8$ &
\begin{minipage}{2.8cm}
\centering
\begin{tikzpicture}[scale=0.6, every node/.style={draw,circle,fill=black,inner sep=1pt}]
\node (1) at (0,0) {}; \node (2) at (1,0) {}; \node (3) at (1,1) {}; \node (4) at (0,1) {}; \node (5) at (2,0) {};
\draw (1)--(2)--(3)--(4)--(1); \draw (1)--(3); \draw (2)--(5);
\end{tikzpicture}
\end{minipage}
& $\{2.641,\, 0.724,\,-0.589,\,-1,\,-1.776\}$ & $6.804$ \\
\hline

$H_9$ &
\begin{minipage}{2.8cm}
\centering
\begin{tikzpicture}[scale=0.6, every node/.style={draw,circle,fill=black,inner sep=1pt}]
\node (1) at (0,0) {}; \node (2) at (1,0) {}; \node (3) at (1,1) {}; \node (4) at (0,1) {}; \node (5) at (2,1) {};
\draw (1)--(2)--(3)--(4)--(1); \draw (1)--(3); \draw (3)--(5);
\end{tikzpicture}
\end{minipage}
& $\{2.686,\, 0.335,\, 0,\,-1.271,\,-1.749\}$ & $7.406$ \\
\hline

$H_{10}$ &
\begin{minipage}{2.8cm}
\centering
\begin{tikzpicture}[scale=0.6, every node/.style={draw,circle,fill=black,inner sep=1pt}]
\node (1) at (0,0) {}; \node (2) at (1,0) {}; \node (3) at (1,1) {}; \node (4) at (0,1) {}; \node (5) at (2,1) {};
\draw (1)--(2)--(3)--(4)--(1); \draw (1)--(3); \draw (3)--(5)--(2);
\end{tikzpicture}
\end{minipage}
& $\{2.935,\, 0.618,\,-0.463,\,-1.473,\,-1.618\}$ & $7.530$ \\
\hline

$H_{11}$ &
\begin{minipage}{2.8cm}
\centering
\begin{tikzpicture}[scale=0.6, every node/.style={draw,circle,fill=black,inner sep=1pt}]
\node (1) at (0,0) {}; \node (2) at (1,0) {}; \node (3) at (1,1) {}; \node (4) at (0,1) {}; \node (5) at (2,0) {}; \node (6) at (1.5,1) {};
\draw (1)--(2)--(3)--(4)--(1); \draw (1)--(3); \draw (5)--(2)--(4); \draw (2)--(6)--(5);
\end{tikzpicture}
\end{minipage}
& $\{3.262,\, 1.340,\,-1,\,-1,\,-1,\,-1.602\}$ & $7.109$ \\
\hline

$H_{12}$ &
\begin{minipage}{2.8cm}
\centering
\begin{tikzpicture}[scale=0.6, every node/.style={draw,circle,fill=black,inner sep=1pt}]
\node (1) at (0,0) {};
\node (2) at (1,0) {};
\node (3) at (0.5,0.866) {}; 
\draw (1)--(2)--(3)--(1);

\node (4) at (-0.5,0.866) {};
\node (5) at (-1,0) {};
\draw (4)--(5);
\draw (4)--(1);
\draw (5)--(1);

\node (6) at (1.5,0.866) {};
\node (7) at (2,0) {};
\draw (6)--(7);
\draw (6)--(2);
\draw (7)--(2);

\node (8) at (0.5,1.4) {};
\draw (8)--(3);
\end{tikzpicture}
\end{minipage}
& $\{2.866,\, 1.732,\, 1,\, -0.211,\, -1,\, -1,\, -1.655,\, -1.732\}$ & $11.742$ \\
\hline

$H_{13}$ &
\begin{minipage}{2.8cm}
\centering
\begin{tikzpicture}[scale=0.45, every node/.style={draw,circle,fill=black,inner sep=1pt}]
\node (v)  at (0,1.5) {}; \node (a)  at (-1,0.5) {}; \node (b)  at (1,0.5) {};
\node (c)  at (0,0.5) {}; \node (d1) at (-2,0.5) {}; \node (d2) at (-3,0.5) {};
\draw (v)--(a)--(c)--(v); \draw (v)--(b); \draw (a)--(d1)--(d2);
\end{tikzpicture}
\end{minipage}
& $\{2.334,\,1.100,\,0.274,\,-0.595,\,-1.374,\,-1.740\}$ & $8.068$ \\
\hline

$H_{14}$ &
\begin{minipage}{2.8cm}
\centering
\begin{tikzpicture}[scale=0.45, every node/.style={draw,circle,fill=black,inner sep=1pt}]
\node (v)  at (0,1.5) {}; \node (a)  at (-1,0.5) {}; \node (b)  at (1,0.5) {};
\node (c)  at (0,0.5) {}; \node (d1) at (-2,0.5) {}; \node (d2) at (-3,0.5) {};
\draw (v)--(a)--(c)--(v); \draw (v)--(b)--(c); \draw (a)--(d1)--(d2);
\end{tikzpicture}
\end{minipage}
& $\{2.655,\,1.211,\,0,\,-1,\,-1,\,-1.866\}$ & $8.499$ \\
\hline

$H_{15}$ &
\begin{minipage}{2.8cm}
\centering
\begin{tikzpicture}[scale=0.45, every node/.style={draw,circle,fill=black,inner sep=1pt}]
\node (v)  at (0,1.5) {}; \node (a)  at (-1,0.5) {}; \node (b)  at (0,0.5) {};
\node (c)  at (1,0.5) {}; \node (d1) at (-2,0.5) {}; \node (d2) at (-3,0.5) {}; \node (e) at (0,-0.5) {};
\draw (v)--(a)--(b)--(v); \draw (v)--(c); \draw (a)--(d1)--(d2); \draw (b)--(e);
\end{tikzpicture}
\end{minipage}
& $\{2.438,\,1.139,\,0.618,\,0,\,-0.820,\,-1.618,\,-1.757\}$ & $10.679$ \\
\hline

$H_{16}$ &
\begin{minipage}{2.8cm}
\centering
\begin{tikzpicture}[scale=0.45, every node/.style={draw,circle,fill=black,inner sep=1pt}]
\node (v)  at (0,1.5) {}; \node (a)  at (-1,0.5) {}; \node (b)  at (1,0.5) {};
\node (c)  at (0,0.5) {}; \node (d1) at (-2,0.5) {}; \node (d2) at (-3,0.5) {}; \node (e) at (0,-0.5) {};
\draw (v)--(a)--(c)--(v); \draw (v)--(b)--(c); \draw (a)--(d1)--(d2); \draw (c)--(e);
\end{tikzpicture}
\end{minipage}
& $\{2.765,\,1.239,\,0.326,\,0,\,-1,\,-1.375,\,-1.955\}$ & $11.074$ \\
\hline

$H_{17}$ &
\begin{minipage}{2.8cm}
\centering
\begin{tikzpicture}[scale=0.45, every node/.style={draw,circle,fill=black,inner sep=1pt}]
\node (v)  at (0,1.5) {}; \node (a)  at (-0.8,0.5) {}; \node (b)  at (0.8,0.5) {};
\node (a1) at (-1.8,0.5) {}; \node (a2) at (-1.8,1.5) {}; \node (b1) at (1.8,0.5) {};
\draw (v)--(a)--(b)--(v); \draw (a1)--(a)--(a2); \draw (b)--(b1);
\end{tikzpicture}
\end{minipage}
& $\{2.445,\, 0.796,\, 0,\,0,\,
 -1.370,\, -1.872\}$ & $9.136$ \\
\hline
\end{tabular}
\caption{The spectrum and the negative $3$-energies of small graphs.}
\label{table:3neg-energy}
\end{table}

\subsection[smallG]{Some graphs $G$ satisfy $\mathcal{E}_3^-(G)\geq \mathcal{E}_3^-(K_n)$}

The following \Cref{lem:K_n-e,lem:K_{n-1}+v,lem:K_(n-2)+K_2,lem:star-plus} can be proved directly by analyzing their quotient matrices corresponding to certain vertex partitions. We present only the proof of the first lemma, and the remaining proofs are provided in the Appendix for the sake of completeness.

\begin{lemma}\label{lem:K_n-e}
Let $n$ be a positive integer with $n\geq 5$. If $G$ is a graph formed from $K_n$ by deleting one edge, then $\mathcal{E}_3^{-}(G)\geq n+1$.
\end{lemma}

\begin{proof}
Assume that $n\geq 5$. Let $V(G)=\{v_1, v_2, \ldots, v_n\}$ and assume that $v_1v_2\not\in E(G)$. We consider a vertex partition $\mathcal{P}$ of $V(G)=\{v_1\}\cup \{v_2\}\cup \{v_3, \ldots, v_n\}$. The adjacency matrix $A(G)$ and the corresponding quotient matrix $M_\mathcal{P}$ of $G/\mathcal{P}$ are as follows:
$$A(G)=
\begin{pmatrix}
0 & 0 & 1 & 1 & \cdots & 1 \\
0 & 0 & 1 & 1 & \cdots & 1 \\
1 & 1 & 0 & 1 & \cdots & 1 \\
1 & 1 & 1 & 0 & \cdots & 1 \\
\vdots & \vdots & \vdots & \vdots & \ddots & \vdots \\
1 & 1 & 1 & 1 & \cdots & 0
\end{pmatrix}
\text{~~and~~} M_\mathcal{P}=
\begin{pmatrix}
0 & 0 & n-2\\
0 & 0 & n-2\\
1 & 1 & n-3
\end{pmatrix}.$$
Observe that for the adjacency matrix $A(G)$, $(0,0,1,-1,0,\ldots,0), (0,0,1,0,-1,\ldots,0), \ldots,$ and $(0,0,1,0,\ldots,0,-1)$ are the eigenvectors of the eigenvalue $-1$ with multiple $n-3$, and $(x,y,z,z,z,\ldots,z)$ are the eigenvectors of other three eigenvalues. So we only consider the quotient matrix $G/\mathcal{P}$.

Note that $\det(\lambda I-M_\mathcal{P})=\lambda^3-(n-3)\lambda^2-(2n-4)\lambda$, which has one zero root and two nonzero roots, denoted by $\lambda_1$ and $\lambda_2$.
It is easy to get that 
$$\lambda_1=\frac{n-3+\sqrt{n^2+2n-7}}{2}>0 \text{~~and~~}
\lambda_2=\frac{n-3-\sqrt{n^2+2n-7}}{2}<0.$$
Thus the negative eigenvalues of $G$ are $-1$ (with multiplicity $n-3$) and $\lambda_2$. 
Then
$$\mathcal{E}_3^-(G)=|-1|^3\times (n-3)+|\lambda_2|^3=n-3+(\frac{\sqrt{n^2+2n-7}-(n-3)}{2})^3 \geq n+1.$$
The last inequality holds because $|\lambda_2|$ is increasing in $n$, and a direct computation shows that $|\lambda_2|^3 > 4$  for all $n \geq 5$.
\end{proof}

\begin{lemma}\label{lem:K_{n-1}+v}
Let $n$ be a positive integer with $n\geq 4$. If $G$ is a graph formed from $K_{n-1}$ by attaching a pendant vertex to one of its vertices, then $\mathcal{E}_3^{-}(G)\geq n$.
\end{lemma}

\begin{lemma}\label{lem:K_(n-2)+K_2}
Let $n$ be a positive integer with $n\geq 6$. If $G$ is a graph formed from $K_{n-2}$ by adding an extra edge $uv$ and connecting $u$ and $v$ to one same vertex of $K_{n-2}$, then $\mathcal{E}_3^{-}(G)>n+1$. 
\end{lemma}

\begin{lemma}\label{lem:star-plus}
Let $n$ be a positive integer. If $G$ is a graph formed from $K_{1,n}$ by subdividing $t$ edges each once, then  $\operatorname{Spec}(G)=\{\pm\sqrt{x_1},\pm\sqrt{x_2},0^{n-t-1},1^{t-1},(-1)^{t-1}\}$, where $x_1$ and $x_2$ are the roots of $x^2-(n+1)x+(n-t)=0$ with $x_1>x_2$. 
\end{lemma}

Above, we computed the eigenvalues of some specific graphs directly from their adjacency matrices. From here onward, we focus on graphs possessing certain favorable structural properties.

\begin{lemma}\label{lem:UnionOfCliques}
Let $G$ be a connected $n$-vertex graph, not isomorphic to $K_n$ or $P_3$. If $G$ contains a vertex $v$ such that $G-v$ is a disjoint union of cliques, then $\mathcal{E}_3^{-}(G) \geq n$.
\end{lemma}

\begin{proof}
Let $G$ be a connected $n$-vertex graph, not isomorphic to $K_n$ or $P_3$. By~\Cref{lem:UpToTenV} we may assume that $n \geq 11$. Assume that $v$ is a vertex of $G$ such that $G-v$ is a disjoint union of cliques $G_1, \ldots, G_\ell$. Note that selecting one vertex $v_i$ from each of $G_i$ for $i\in \{1,\ldots, \ell\}$, together with the vertex $v$, yields an induced subgraph of $G$ isomorphic to $K_{1,\ell}$. By~\Cref{lem:complete graph+star+cycle}\ref{K1n}, $\operatorname{Spec}(K_{1,\ell})=\{\sqrt{\ell}, 0^{\ell-1},-\sqrt{\ell}\}$. Thus $|\lambda_n(G)| \geq |\lambda_{\ell+1}(K_{1,\ell})|= \sqrt{\ell}$. Since $G-v$ is a disjoint union of $\ell$ cliques, by~\Cref{lem:complete graph+star+cycle}\ref{Kn}, $\mathcal{E}_3^{-}(G-v)\geq n-\ell-1$. By~\Cref{thm:alternate} (the Interlacing Theorem), we have that
$$\mathcal{E}_3^{-}(G) \geq \mathcal{E}_3^{-}(G-v) - |\lambda_{n-1}(G - v)|^3 + |\lambda_n(G)|^3\geq(n-\ell-1)-1+\ell^{\frac{3}{2}}.$$ If $\ell\geq 3$, then we immediately get $\mathcal{E}_{3}^{-}(G)\geq n+0.19>n$. It remains to consider whether $\ell=1$ or $\ell=2$.

\noindent
\textbf{Case 1:} $\ell=1$. Since $G$ is not isomorphic to $K_n$, there is at least one vertex of $G-v$ that is not adjacent to $v$. As $n\geq 11$, either $G$ is isomorphic to $K_n-e$ (i.e., exactly one vertex in $G-v$ is not adjacent to $v$), or $G$ contains an induced subgraph isomorphic to $H_1$, as shown in~\Cref{fig:H2}. 
\begin{figure}[ht]
\centering
\begin{tikzpicture}[every node/.style={draw,circle,fill=black,inner sep=1.5pt}]
\begin{scope}[xshift=4cm, rotate=0]
\node (2) at (1,1.2) {}; 
\node (3) at (2,0.6) {}; 
\node (4) at (1,0) {}; 
\node (5) at (3,0.6) {};
\draw (2)--(3)--(4); 
\draw (2)--(4); 
\draw (3)--(5);
\node[draw=none, fill=none, below=2pt] at (5) {$v$};
\end{scope}
\end{tikzpicture}
\caption{The configuration $H_1$ for Case 1 in \Cref{lem:UnionOfCliques}}
\label{fig:H2}
\end{figure}
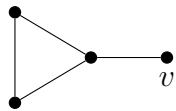 

The first case has been verified in~\Cref{lem:K_n-e}. For the second case, noting that  $\lambda_{\mathrm{min}}(H_1)=\lambda_{4}(H_1)\approx -1.481$ by~\Cref{table:3neg-energy}, again by~\Cref{thm:alternate},
$$\mathcal{E}_{3}^{-}(G) 
\geq \mathcal{E}_{3}^{-}(G - v) - |\lambda_{n-1}(G - v)|^3 + |\lambda_n(G)|^3 \geq (n - 2) - 1 + |\lambda_{\mathrm{min}}(H_1)|^3
= n + 0.24>n.$$

\noindent
\textbf{Case 2}. $\ell=2$. Assume that $G-v$ consists of two disjoint cliques $G_1$ and $G_2$. Without loss of generality, we assume that $|G_1|\geq |G_2|$. Noting that $n\geq 11$, $|G_1|\geq 5$ and $|G_2|\geq 1$. 

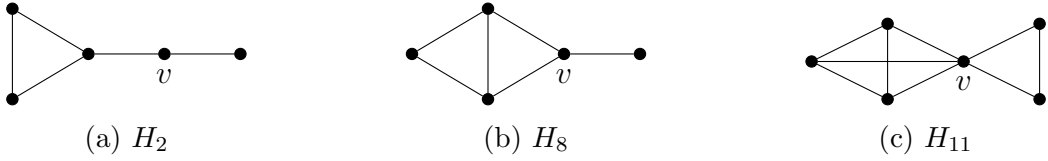
\begin{figure}[ht]
\centering
\begin{subfigure}[t]{.32\textwidth}
\centering
\begin{tikzpicture}[every node/.style={draw,circle,fill=black,inner sep=1.5pt}]
\node (2) at (1,1.2) {}; 
\node (3) at (2,0.6) {}; 
\node (4) at (1,0) {}; 
\node (5) at (3,0.6) {};
\node (6) at (4,0.6) {}; 
\draw (2)--(3)--(4); 
\draw (2)--(4); 
\draw (3)--(5)--(6);
\node[draw=none, fill=none, below=2pt] at (5) {$v$};
\end{tikzpicture}
\caption{$H_2$}
\label{fig:H2inLem3.6}
\end{subfigure}
\begin{subfigure}[t]{.32\textwidth}
\centering
\begin{tikzpicture}[every node/.style={draw,circle,fill=black,inner sep=1.5pt}]
\node (1) at (0,0.6) {}; 
\node (2) at (1,1.2) {}; 
\node (3) at (2,0.6) {}; 
\node (4) at (1,0) {}; 
\node (5) at (3,0.6) {}; 
\draw (1)--(2)--(3)--(4)--(1); 
\draw (2)--(4); 
\draw (3)--(5);
\node[draw=none, fill=none, below=2pt] at (3) {$v$};
\end{tikzpicture}
\caption{$H_{8}$}
\label{fig:H8inLem3.6}
\end{subfigure}
\begin{subfigure}[t]{.32\textwidth}
\centering
\begin{tikzpicture}[every node/.style={draw,circle,fill=black,inner sep=1.5pt}]
\node (1) at (1,0) {};
\node (2) at (0,0.5) {};
\node (5) at (2,0.5) {};
\node (4) at (1,1) {};
\node (7) at (3,1) {};
\node (8) at (3,0) {};
\draw (1) -- (2) ;
\draw (1) -- (4);
\draw (2) -- (4);
\draw (1) -- (5);
\draw (2) -- (5);
\draw (4) -- (5);
\draw (7) -- (5);
\draw (8) -- (5);
\draw (8) -- (7);
\node[draw=none, fill=none, below=2pt] at (5) {$v$};
\end{tikzpicture}
\caption{$H_{11}$}
\label{fig:H11inLem3.6}
\end{subfigure}
\caption{Configurations $H_2$, $H_8$, and $H_{11}$ for Case 2 in \Cref{lem:UnionOfCliques}}
\label{fig:Case2inLem3.7}
\end{figure}

If $v$ has at most $3$ neighbors in $G_1$, then there are at least two vertices in $G_1$ that are not adjacent to $v$. Hence, $G$ contains an induced subgraph isomorphic to $H_2$, as shown in~\Cref{fig:H2inLem3.6}. 
By~\Cref{table:3neg-energy}, $\lambda_{\mathrm{min}}(H_2)=\lambda_{5}(H_2)\approx -1.675$, thus we have that $|\lambda_{\mathrm{min}}(H_2)|^{3}> 4$. By~\Cref{thm:alternate},
$$\mathcal{E}_{3}^{-}(G) \geq\mathcal{E}_{3}^{-}(G-v)-|\lambda_{n-1}(G-v)|^{3}+ |\lambda_{n}(G)|^{3}\geq(n-3)-1+|\lambda_{5}(H_2)|^{3}>n.$$

Now we consider the case when $v$ has at least $4$ neighbors in $G_1$. Recall that $|G_1|\geq 5$ and $|G_2|\geq 1$. We consider the following two possibilities.

Assume that $|G_2|=1$. If there is a vertex of $G_1$ not adjacent to $v$, since $|G_1|\geq 5$, then $G$ contains an induced graph isomorphic to $H_8$, as shown in~\Cref{fig:H8inLem3.6}. By~\Cref{table:3neg-energy}, $\lambda_{\mathrm{min}}(H_8)=\lambda_{5}(H_8)=-1.776$, thus we have that $|\lambda_{\mathrm{min}}(H_8)|^{3}> 4$. By the Interlacing Theorem (\Cref{thm:alternate}),
$$\mathcal{E}_{3}^{-}(G) \geq\mathcal{E}_{3}^{-}(G-v)-|\lambda_{n-1}(G-v)|^{3}+ |\lambda_{n}(G)|^{3}\geq(n-3)-1+|\lambda_{5}(H_8)|^{3}>n.$$
Thus, all the vertices of $G_1$ are adjacent to $v$. So $G$ must be isomorphic to $K_{n-1}$ attached with a pendant vertex (which is exactly $v$), in which case by~\Cref{lem:K_{n-1}+v}, $\mathcal{E}_{3}^{-}(G) \geq n$.

Assume that $|G_2|\geq 2$. In this case, if there is a vertex of $G_2$ not adjacent to $v$, then $G$ contains $H_2$ as an induced graph, we are done. So now $v$ is adjacent to at least two vertices of $G_2$, in which case $G$ contains an induced subgraph isomorphic to $H_{11}$, as shown in~\Cref{fig:H11inLem3.6}. By~\Cref{table:3neg-energy}, $|\lambda_{\mathrm{min}}(H_{11})|^{3}=|\lambda_{6}(H_{11})|^{3}\geq 1.602^3 >4$. By~\Cref{thm:alternate}, 
$$\mathcal{E}_{3}^{-}(G) \geq\mathcal{E}_{3}^{-}(G-v)-|\lambda_{n-1}(G-v)|^{3}+ |\lambda_{n}(G)|^{3}\geq (n-3)-1+|\lambda_{6}(H_{11})|^{3}>n.$$
We complete the proof of the lemma.
\end{proof}

\begin{lemma}\label{lem:oneP3+onecliques}
Let $G$ be a connected $n$-vertex graph. Then $\mathcal{E}_3^{-}(G) \geq n$ if the following conditions hold: there is a vertex $u\in V(G)$ such that $V(G)\setminus \{u\}$ admits a partition $\{X_1, X_2\}$ such that $G[X_1]$ is an induced $P_3$, $G[X_2]$ is a clique, and there is a vertex $v\in N(u)\cap X_2$ with $N(v)\cap X_1=\emptyset$.  
\end{lemma}

\begin{proof}
Assume that there is a vertex $u\in V(G)$ such that $V(G)\setminus \{u\}$ admits a partition $\{X_1, X_2\}$ such that $G[X_1]$ is an induced $P_3$, $G[X_2]$ is a clique, and there is a vertex $v\in N(u)\cap X_2$ with $N(v)\cap X_1=\emptyset$. Note that $H:=G[X_1\cup\{u,v\}]$ has five vertices and is isomorphic to one of $P_5$, $H_3$, $H_4$, $H_6$, and $H_9$, as depicted in~\Cref{fig:oneQ_1+oneC_1}. 
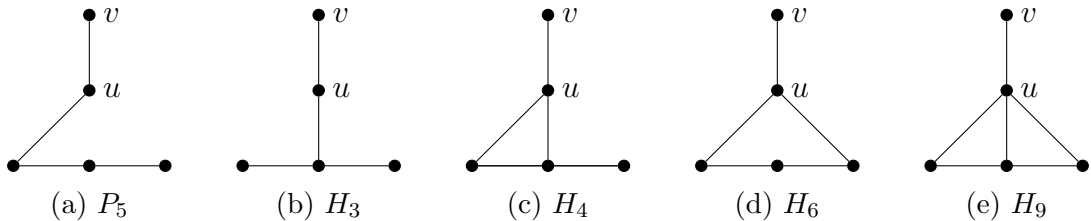
\begin{figure}[htbp]
\centering 
\begin{subfigure}[t]{.18\textwidth}
\centering  
\begin{tikzpicture}[every node/.style={draw,circle,fill=black,inner sep=1.5pt}]
\node (1) at (0,-1) {};
\node (2) at (1,-1) {};
\node (3) at (2,-1) {};
\node (4) at (1,0) {}; 
\node (5) at (1,1) {}; 
\draw (1)--(2)--(3);
\draw (1)--(4)--(5);
\node[draw=none, fill=none, right=2pt] at (4) {$u$};
\node[draw=none, fill=none, right=2pt] at (5) {$v$};
\end{tikzpicture}
\caption{$P_5$}
\label{fig:2} 
\end{subfigure}
\begin{subfigure}[t]{.18\textwidth}
\centering  
\begin{tikzpicture}[every node/.style={draw,circle,fill=black,inner sep=1.5pt}]
\node (1) at (0,-1) {};
\node (2) at (1,-1) {};
\node (3) at (2,-1) {};
\node (4) at (1,0) {}; 
\node (5) at (1,1) {}; 
\draw (1)--(2)--(3);
\draw (2)--(4)--(5);
\node[draw=none, fill=none, right=2pt] at (4) {$u$};
\node[draw=none, fill=none, right=2pt] at (5) {$v$};
\end{tikzpicture}
\caption{$H_3$}
\label{fig:1} 
\end{subfigure}
\begin{subfigure}[t]{.18\textwidth}
\centering  
\begin{tikzpicture}[every node/.style={draw,circle,fill=black,inner sep=1.5pt}]
\node (1) at (0,-1) {};
\node (2) at (1,-1) {};
\node (3) at (2,-1) {};
\node (4) at (1,0) {}; 
\node (5) at (1,1) {}; 
\draw (1)--(2)--(3);
\draw (2)--(4)--(1);
\draw (4)--(5);
\draw (1)--(3);
\node[draw=none, fill=none, right=2pt] at (4) {$u$};
\node[draw=none, fill=none, right=2pt] at (5) {$v$};
\end{tikzpicture}
\caption{$H_4$}
\label{fig:3} 
\end{subfigure}
\begin{subfigure}[t]{.18\textwidth}
\centering  
\begin{tikzpicture}[every node/.style={draw,circle,fill=black,inner sep=1.5pt}]
\node (1) at (0,-1) {};
\node (2) at (1,-1) {};
\node (3) at (2,-1) {};
\node (4) at (1,0) {}; 
\node (5) at (1,1) {}; 
\draw (1)--(2)--(3);
\draw (1)--(4)--(5);
\draw (3)--(4);
\node[draw=none, fill=none, right=2pt] at (4) {$u$};
\node[draw=none, fill=none, right=2pt] at (5) {$v$};
\end{tikzpicture}
\caption{$H_6$}
\label{fig:4} 
\end{subfigure}
\begin{subfigure}[t]{.18\textwidth}
\centering
\begin{tikzpicture}[every node/.style={draw,circle,fill=black,inner sep=1.5pt}]
\node (1) at (0,-1) {};
\node (2) at (1,-1) {};
\node (3) at (2,-1) {};
\node (4) at (1,0) {}; 
\node (5) at (1,1) {}; 
\draw (1)--(2)--(3);
\draw (1)--(4)--(5);
\draw (2)--(4);
\draw (3)--(4);
\node[draw=none, fill=none, right=2pt] at (4) {$u$};
\node[draw=none, fill=none, right=2pt] at (5) {$v$};
\end{tikzpicture}
\caption{$H_9$}
\label{fig:5}
\end{subfigure}
\caption{The possibilities of $H$ in \Cref{lem:oneP3+onecliques}}
\label{fig:oneQ_1+oneC_1}
\end{figure}  

By~\Cref{table:3neg-energy}, $\mathcal{E}_3^-(H)\geq \min\{\mathcal{E}_3^-(P_5),\mathcal{E}_3^-(H_3),\mathcal{E}_3^-(H_4),\mathcal{E}_3^-(H_6),\mathcal{E}_3^-(H_9)\}\geq 6$. Note that $G-H$ is a clique with $n-5$ vertices, and thus by~\Cref{lem:complete graph+star+cycle}\ref{Kn} $\mathcal{E}_3^-(G-H)=n-6$. By~\Cref{thm: additivity}, $\mathcal{E}_3^-(G)\geq \mathcal{E}_3^-(G-H)+\mathcal{E}_3^-(H)\geq (n-6)+6=n$.
\end{proof}

\begin{corollary}\label{cor:oneP3+onecliques}
Let $G$ be a connected $n$-vertex graph. If $G$ contains a vertex $u$ such that $G-u$ is a disjoint union of exactly one $P_3$ and one clique, then $\mathcal{E}_3^{-}(G) \geq n$.
\end{corollary}

\begin{lemma}\label{lem:P3+cliques}
Let $G$ be a connected $n$-vertex graph. If $G$ contains a vertex $v$ such that $G-v$ is a disjoint union of some $P_3$'s and some cliques, then $\mathcal{E}_3^{-}(G) \geq n$.
\end{lemma}

\begin{proof}
Let $G$ be an $n$-vertex graph and let $v$ be a vertex such that $G-v$ is a disjoint union of some $P_3$'s and some cliques. By~\Cref{lem:UnionOfCliques}, we may assume that $G-v$ contains at least one $P_3$.
Observing that $G$ is not isomorphic to $K_n$ or $P_3$, by \Cref{lem:UpToTenV}, $n \geq 11$. 

Let $Q_1, \ldots, Q_{\ell_1}$ denote the copies of $P_3$'s in $G-v$ such that each $|V(Q_i)\cap N(v)|\leq 2$ for $i\in \{1,\ldots, \ell_1\}$. 
Let $C_1, \ldots, C_{\ell_2}$ denote the cliques or copies of $P_3$'s (in $G-v$) whose all three vertices are adjacent to $v$. Note that $G$ contains an induced subgraph $T$, rooted at $v$, which is isomorphic to the graph formed from the star $K_{1,\ell_1+\ell_2}$ by subdividing exactly $\ell_1$ of its edges once each. See~\Cref{fig:T} for an illustration. 
\begin{figure}[ht]
\centering
\begin{tikzpicture}[every node/.style={draw,circle,fill=black,inner sep=1.5pt}, xscale=1.8]
\node (v) at (4.2,3.5) {};
\node[draw=none, fill=none, above=2pt] at (v) {$v$};
\node (u1)  at (1,2.5) {};
\node (u2)  at (2,2.5) {};
\node (u3)  at (3,2.5) {};
\node (u4)  at (4,2.5) {};
\node (u5)  at (5,2.5) {};
\node (u6)[fill=gray!10]  at (5.5,2.5) {};
\node (u7)[fill=gray!10]  at (6,2.5) {};
\node (u8) at (6.5,2.5) {};
\node (u9)  at (7,2.5) {};
\node (u10)[fill=gray!10] at (7.5,2.5) {};

\draw (v)--(u1);
\draw (v)--(u2);
\draw (v)--(u3);
\draw (v)--(u4);
\draw (v)--(u5);
\draw (v)--(u6);
\draw (v)--(u7);
\draw (v)--(u8);
\draw (v)--(u9);
\draw (v)--(u10);
\node (w11) at (1,1.5) {}; 
\node[fill=gray!10](w12) at (1,0.5) {};
\node (w21) at (1.8,1.5) {}; 
\node[fill=gray!10] (w22) at (2.2,1.5) {};
\node (w31) at (3,1.5) {}; 
\node (w32)[fill=gray!10] at (3.5,2.5) {};
\node (w41) at (4.5,2.5)[fill=gray!10] {}; 
\node (w42) at (4.25,1.5) {};
\node (w81)[fill=gray!10] at (6.25,1.5) {}; \node[fill=gray!10] (w82) at (6.75,1.5) {};
\node (w91)[fill=gray!10] at (7,1.5) {}; \node[fill=gray!10] (w102) at (7.5,1.5) {};
\draw (v)--(w32);
\draw (v)--(w41);
\draw (u1)--(w11);  
\draw (w11)--(w12);
\draw (u2)--(w21);  
\draw (u2)--(w22);
\draw (u3)--(w31);  
\draw (u3)--(w32);
\draw (w42)--(w41);  
\draw (u4)--(w42);
\draw (u5)--(u6)--(u7); 
\draw (u8)--(w81);  
\draw (u8)--(w82);
\draw (u9)--(w91); 
\draw (w81)--(w82);
\draw (w91)--(u10)--(u9);
\draw (u10)--(w102)--(u9);
\draw (w91)--(w102);
\end{tikzpicture}
\caption{An example of $T$ induced on the vertices in black}
\label{fig:T}
\end{figure}
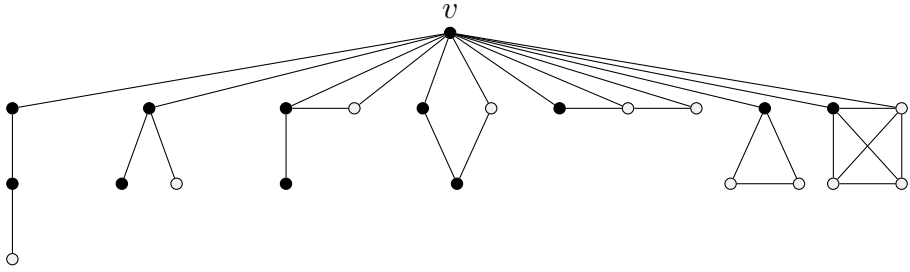
Note that $T$ has exactly $2\ell_1+\ell_2+1$ vertices. By~\Cref{lem:star-plus}, $Spec(T)=\{\pm\sqrt{x_1},\pm\sqrt{x_2},0^{\ell_2-1},1^{\ell_1-1},(-1)^{\ell_1-1}\}$, where $x_1$ and $x_2$ are the two roots of $x^2-(\ell_1+\ell_2+1)x+\ell_2=0$ with $x_1>x_2$. 
By the Interlacing Theorem (\Cref{thm:alternate}), we have that $\lambda_n\leq -\sqrt{x_1}\leq\lambda_{n-1}\leq-\sqrt{2}\leq \lambda_{n-2}\leq\dots\leq\lambda_{n-\ell_1+1}\leq-\sqrt{2}\leq\lambda_{n-\ell_1}\leq-1\leq\lambda_{n-\ell_1-1}\leq \dots\leq \lambda_{2\ell_1+\ell_2+2}\leq -1$. 
The first inequality follows from the fact that $G$ contains $T$ as an induced subgraph with $\lambda_{\ell_1+\ell_2+1}(T)=\sqrt{x_1}$, and all the other inequalities follow from the fact that $G-v$ contains a copy of $P_3$ or a clique as induced subgraphs. Then we have that  $${\small \begin{aligned}\mathcal{E}_3^{-}(G)&\geq |-\sqrt{x_1}|^3+(\ell_1-1)\times|-\sqrt{2}|^3+(n-3\ell_1-\ell_2-1)\times|-1|^3\\
&=
(\frac{\,\ell_1+\ell_2+1+\sqrt{(\ell_1+\ell_2-1)^2+4\ell_1}}{2})^\frac{3}{2}+2\sqrt{2}(\ell_1-1)+(n-3\ell_1-\ell_2-1)=:f(\ell_1, \ell_2).
\end{aligned}}$$
Observe that the term involving $\ell_1$ in the function $f(\ell_1, \ell_2)$ is at least $(\sqrt{\ell_1}+2\sqrt{2}-3)\ell_1$ and thus $f(\ell_1, \ell_2)$ is increasing as $\ell_1$ increases when $\ell_1\geq 1$.

\medskip
\noindent
{\bf Case (i).} \emph{Assume that $\ell_1\geq 2$}.
\medskip

If $\ell_1\geq 2$, then $f(\ell_1, \ell_2)\geq f(2, \ell_2)$ where $$
f(2, \ell_2)=(\frac{\,3+\ell_2+\sqrt{(\ell_2-1)^2+8}}{2})^\frac{3}{2}+2\sqrt{2}+(n-\ell_2-7).$$ Note that the term involving $\ell_2$ in $f(2, \ell_2)$ is at least $(\sqrt{\ell_2}-1)\ell_2$ and thus $f(2, \ell_2)$ is increasing as $\ell_2$ increases when $\ell_2\geq 2$. Thus if $\ell_1\geq 1$ and $\ell_2\geq 2$, then $f(\ell_1, \ell_2)\geq f(2,2)=8+2\sqrt{2}-9+n\approx 1.8+n\geq n$. For small cases, by computation, $f(2,0)=3\sqrt{3}+2\sqrt{2}-7+n\approx 1.02+n$, and $f(2,1)=(2+\sqrt{2})^{\frac{3}{2}}+2\sqrt{2}-8+n\approx 1.13+n$. Therefore, in this case, $\mathcal{E}_3^-(G)\geq f(2, \ell_2)\geq n$.

\medskip
\noindent
{\bf Case (ii).} \emph{Assume that $\ell_1=1$}.
\medskip

If $\ell_1=1$ and $\ell_2=1$, then as $n\geq 11$, $C_1$ must be a clique. Thus $\mathcal{E}_3^-(G)\geq n$ by~\Cref{cor:oneP3+onecliques}. Next we consider when $\ell_2\geq 2$. We have that $$f(1, \ell_2)=(\frac{2+\ell_2+\sqrt{{\ell_2}^2+4}}{2})^{\frac{3}{2}}+(n-4-\ell_2).$$ Note that the term involving $\ell_2$ in $f(1, \ell_2)$ is at least $(\sqrt{\ell_2+1}-1)\ell_2$ and thus $f(1, \ell_2)$ is increasing as $\ell_2$ increases (as $\ell_2\geq 2$). Therefore, in this case, $\mathcal{E}_3^-(G)\geq f(1, \ell_2)\geq f(1, 2)=(2+\sqrt{2})^{\frac{3}{2}}-6+n\approx0.30+n > n$.

\medskip
\noindent
{\bf Case (iii).} \emph{Assume that $\ell_1=0$}.
\medskip

If $\ell_1=0$, since $n\geq 11$, then $\ell_2\geq 1$. Moreover, we know that there exists at least one copy of $P_3$ in $G-v$ such that all its vertices are adjacent to $v$. Let $p$ denote the number of such $P_3$'s. Thus the number of cliques in $G-v$ is $\ell_2-p$.

Note that $G$ contains an induced subgraph $H'$ isomorphic to a star $K_{1,p+\ell_2}$, as it consists of $v$ together with one vertex from each of cliques in $G-v$, and two endpoints of each of $P_3$'s. By~\Cref{lem:complete graph+star+cycle}\ref{K1n},
$\mathcal{E}_{3}^{-}(H')=\sqrt{p+\ell_2}$. Moreover, $G-H'$ is a disjoint union of cliques, with the number of such cliques being at most $\ell_2$. Noting that $G-H'$ has $n-1-p-\ell_2$ vertices, by~\Cref{lem:complete graph+star+cycle}\ref{Kn}, $\mathcal{E}_{3}^{-}(G-H')\geq (n-1-p-\ell_2)-\ell_2$. It follows from~\Cref{thm: additivity} that $$\mathcal{E}_{3}^{-}(G) \geq \mathcal{E}_{3}^{-}(H')+\mathcal{E}_{3}^{-}(G-H') \geq (\sqrt{p+\ell_2})^3 + (n-1-p-2\ell_2).$$ If $p+\ell_2\geq 4$, since $p\geq 1$, then $\mathcal{E}_{3}^{-}(G)\geq (\sqrt{p+\ell_2})^3 + (n-1-p-\ell_2)-t\geq 2(p+\ell_2) + (n-1-p-2\ell_2)\geq n$. 

It remains to consider when $p+\ell_2\leq 3$. Since $n\geq 11$, $\ell_2\geq 2$; together with the fact $p\geq 1$, we only need to consider the case when $p=1$ and $\ell_2=2$. In this case, $G-v$ is the disjoint union of exactly one $P_3$ and one complete graph; by~\Cref{cor:oneP3+onecliques}, $\mathcal{E}_3^-(G)\geq n$.
\end{proof}

A vertex of a graph is said to be \emph{dominating} if it is adjacent to all the other vertices of this graph.

\begin{lemma}\label{lem:dominating}
Let $G$ be a connected $n$-vertex graph, not isomorphic to $K_n$ or $P_3$. If $G$ contains a dominating vertex $v$, then $\mathcal{E}_3^{-}(G) \geq n$.
\end{lemma}

\begin{proof}
Since $G$ is not isomorphic to $K_n$ or $P_3$, by \Cref{lem:UpToTenV} we may assume that $n \geq 11$. Let $v$ be a dominating vertex of $G$. We proceed by induction on $n$.

If $G-v$ is a disjoint union of cliques, then we are done by~\Cref{lem:UnionOfCliques}. If $G-v$ has a connected component with at least $4$ vertices that is not a clique, then it must contain an induced subgraph $H$ not isomorphic to $K_4$ of order $4$. By~\Cref{lem:UpToTenV}, $\mathcal{E}_{3}^{-}(H)\geq 4$. Note that $v$ is also a dominating vertex of $G-H$. By the induction hypothesis, $\mathcal{E}_{3}^{-}(G-H)\geq n-4$. Therefore, by \Cref{thm: additivity}, $\mathcal{E}_{3}^{-}(G) \geq \mathcal{E}_{3}^{-}(H)+\mathcal{E}_{3}^{-}(G-H) > 4 + (n-4) = n.$ Thus each of non-complete connected components of $G-v$ is isomorphic to $P_3$. It then follows from~\Cref{lem:P3+cliques} that $\mathcal{E}_3^{-}(G) \geq n$.
\end{proof}

\begin{lemma}\label{lem:ThreeCliques}
Let $G$ be a connected $n$-vertex graph, not isomorphic to $K_n$ or $P_3$. If $G$ consists of three vertex disjoint cliques $G_1, G_2,$ and $G_3$ with $|G_i|\geq 2$ for $i\in [3]$, and that for each $i\in [3]$ there exists exactly one vertex $x_i \in V(G_i)$ such that $x_1x_2x_3$ forms a triangle, then $\mathcal{E}_3^{-}(G) \geq n$.
\end{lemma}

\begin{proof}
Let $G_1, G_2,$ and $ G_3 $ denote three disjoint cliques with $|G_i|\geq 2$ for each $i\in [3]$, and let $x_1x_2x_3$ denote the triangle. Observing that $G$ is not isomorphic to $K_n$ or $P_3$, by \Cref{lem:UpToTenV} we may assume that $n \geq 11$. Thus we have that $|G_i|\geq 3$ for some $i\in \{1,2,3\}$.

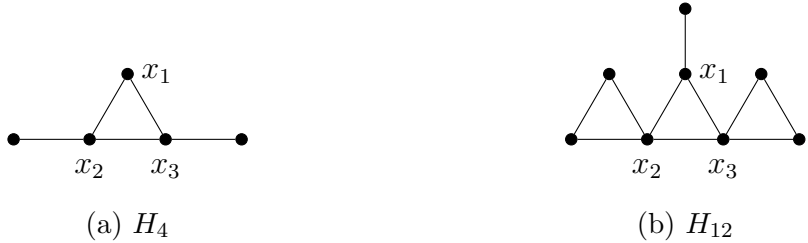
\begin{figure}[htbp]
\begin{subfigure}[t]{.45\textwidth}
\centering  
\begin{tikzpicture}[every node/.style={draw,circle,fill=black,inner sep=1.5pt}]
\node (1) at (0,0) {};
\node (2) at (1,0) {};
\node (3) at (0.5,0.866) {};
\draw (1)--(2)--(3)--(1);

\node (4) at (-1,0) {};
\node (5) at (2,0) {};
\draw (1)--(4);
\draw (2)--(5);
\node[draw=none, fill=none,right=2pt] at (3) {$x_1$};
\node[draw=none, fill=none, below=2pt] at (1) {$x_2$};
\node[draw=none, fill=none, below=2pt] at (2) {$x_3$};
\end{tikzpicture}
\caption{$H_4$}
\label{fig:H3} 
\end{subfigure}
\begin{subfigure}[t]{.45\textwidth}
\centering
\begin{tikzpicture}[every node/.style={draw,circle,fill=black,inner sep=1.5pt}]
\node (1) at (0,0) {};
\node (2) at (1,0) {};
\node (3) at (0.5,0.866) {};
\draw (1)--(2)--(3)--(1);

\node (4) at (-0.5,0.866) {};
\node (5) at (-1,0) {};
\draw (4)--(5);
\draw (4)--(1);
\draw (5)--(1);

\node (6) at (1.5,0.866) {};
\node (7) at (2,0) {};
\draw (6)--(7);
\draw (6)--(2);
\draw (7)--(2);

\node (8) at (0.5,1.732) {};
\draw (8)--(3);
\node[draw=none, fill=none, right=2pt] at (3) {$x_1$};
\node[draw=none, fill=none, below=2pt] at (1) {$x_2$};
\node[draw=none, fill=none, below=2pt] at (2) {$x_3$};
\end{tikzpicture}
\caption{$H_{12}$}
\label{fig:H12}
\end{subfigure}
\caption{Configurations $H_4$ and $H_{12}$ in~\Cref{lem:ThreeCliques}}
\label{fig:K3-cliques} 
\end{figure}

If there is exactly one $G_i$ has at least three vertices, say $G_1$, then $|G_2|=|G_3|=2$ and thus $G$ is the union of $H$ isomorphic to $H_4$ (as shown in~\Cref{fig:H3}) and a clique $K_\ell$ with $\ell=n-|H_{4}|$. By~\Cref{table:3neg-energy}, we have that $\mathcal{E}_3^{-}(H_{4})\approx 6.447\geq 6$, and it follows from~\Cref{thm: additivity} that $$\mathcal{E}_3^{-}(G)\geq \mathcal{E}_3^{-}(H)+\mathcal{E}_3^{-}(G-H)=\mathcal{E}_3^{-}(H_4)+\mathcal{E}_3^{-}(K_{\ell})\geq 6+(n-5-1)=n.$$ 

If there are at least two $G_i$'s with at least three vertices, say $G_2$ and $G_3$, then $G$ contains an induced subgraph $H'$ isomorphic to $H_{12}$ (as shown in~\Cref{fig:H12}), and, moreover, $G-H'$ is a disjoint union of at most three cliques. By~\Cref{table:3neg-energy}, we have that $\mathcal{E}_3^{-}(H_{12})\approx 11.742\geq 11$. Therefore, by~\Cref{thm: additivity}, $$\mathcal{E}_3^{-}(G)\geq \mathcal{E}_3^{-}(H')+\mathcal{E}_3^{-}(G-H')\geq \mathcal{E}_3^{-}(H_{12})+(n-|H_{12}|-3)\geq 11+(n-8-3)=n.$$
It completes the proof of the lemma.
\end{proof}

\subsection[negative-main]{Proof of \Cref{thm:3-energy-main}}\label{sec:mainproof}

\noindent
{\bf A proof sketch}. We proceed by contradiction. Assume that $G$ is a connected $n$-vertex graph, not isomorphic to $K_n$ or $P_3$, such that $\mathcal{E}_3^{-}(G) < n$, and that $n$ is minimum with this property. By a computer-assisted checking, we may assume $n \ge 11$. We choose a vertex $v$ of maximum degree in $G$; it is straightforward to show that $\deg(v) \ge 3$. We can see that $G-N[v]$ is non‑empty and that every connected component of $G-N[v]$ is a clique. Moreover, for each connected component $Q$ of $G-N[v]$, the subgraph induced by $V(Q) \cup N(Q)$ is either a clique or a $P_3$. Our strategy is the following: We attempt to find a ``good" induced connected subgraph $H$ such that: either its negative $3$-energy is large enough to offset the loss in $\mathcal{E}_3^-(G-H)$ (compared to $|G-H|$) caused by disconnection or clique components; or if $G-H$ (which typically contains $v$) is isomorphic to $P_3$, then the degree of $v$ must be bounded, which in turn implies that $G$ has fewer than $10$ vertices. This structural description forces $G-v$ to be a disjoint union of some $P_3$'s and some cliques. By a previously-established lemma (\Cref{lem:P3+cliques}), we have that $\mathcal{E}_3^{-}(G) \ge n$, contradicting the assumption $\mathcal{E}_3^{-}(G) < n$. Hence, no such counterexample exists, completing the proof.

\medskip
\noindent
{\bf \Cref*{thm:3-energy-main}}. \emph{Let $n$ be a positive integer. If $G$ is a connected $n$-vertex graph other than $K_n$ or $P_3$, then $\mathcal{E}_3^{-}(G) \geq n.$}

\begin{proof}
Assume to the contrary that $G$ is a connected $n$-vertex graph, not isomorphic to $K_n$ or $P_3$, but satisfies $\mathcal{E}_3^{-}(G) < n.$ Among all such graphs, we choose $G$ such that $n$ is the smallest integer for which the property holds. By~\Cref{lem:UpToTenV}, we know that $n\geq 11$.

\begin{claim}\label{claim:connected-components}
Let $G'$ be a proper subgraph of $G$ and let $c(G')$ denote the number of connected components of $G'$. Then $\mathcal{E}_3^{-}(G') \geq |G'|-c(G')$.
\end{claim}

\begin{proof*}
Let $G'_1,G'_2,\ldots, G'_{c(G')}$ be the connected components of $G'$. Note that each $G'_i$ is also a proper subgraph of $G$. For any $G'_i$ isomorphic to $P_3$ or being a clique, by~\Cref{table:3neg-energy} and \Cref{lem:complete graph+star+cycle}\ref{Kn}, $\mathcal{E}_3^{-}(G'_i)\geq |G'_i|-1$; If $G'_j$ is neither isomorphic to $P_3$ nor being a clique, then by the minimality of $G$, $\mathcal{E}_3^{-}(G'_j)\geq |G'_j|$. Hence, $\mathcal{E}_3^{-}(G')=\mathcal{E}_3^{-}(G'_1)+\cdots +\mathcal{E}_3^{-}(G'_{c(G')})\geq |G'_1|+\cdots +|G'_{c(G')}|-c(G')=|G'|-c(G').$
\end{proof*}  

\medskip
By the minimality of $G$, the next claim is a direct consequence of~\Cref{thm: additivity}. 

\begin{claim}\label{claim:H+G-H}
Let $H$ be a connected induced subgraph of $G$. If $G-H$ is connected, then at least one of $H$ and $G-H$ is either a clique or isomorphic to $P_3$. 
\end{claim} 

Let $\Delta$ denote the maximum degree of the minimum counterexample $G$. If $\Delta=2$, then $G$ is isomorphic to either $P_n$ or $C_n$. Since $n \geq 11$, for either $P_n$ or $C_n$, $G$ contains an induced subgraph $P_4$ such that $G-P_4$ is isomorphic to $P_{n-4}$, contradicting~\Cref{claim:H+G-H}.

Therefore, we assume that $\Delta \geq 3$. Let $v$ be a vertex of $V(G)$ such that $d_G(v) = \Delta$. Let $N(v)$ denote the set of neighbors of $v$ in $G$ and let $N[v]=N(v)\cup \{v\}$. We know that $G-v$ is not a disjoint union of cliques, as otherwise by \Cref{lem:UnionOfCliques} we have that $\mathcal{E}_3^{-}(G) \geq n$, a contradiction. Furthermore, we have that $G-N[v]$ is not empty (i.e., $v$ is not a dominating vertex), as otherwise by~\Cref{lem:dominating}, $\mathcal{E}_3^{-}(G) \geq n$, again a contradiction.

\begin{claim}\label{claim:G-N[v]isUnionOfCliques}
$G-N[v]$ is a disjoint union of cliques. 
\end{claim}

\begin{proof*}
Assume to the contrary that in $G-N[v]$ there is a connected component $C$ that is not a clique. Since $G$ is connected, there exists a vertex, say $u$, of $N(v)$ adjacent to some vertex of $V(C)$. Note that $|C|\geq 3$, as otherwise, it can only be $K_1$ or $K_2$, which is a clique. Since $\Delta \geq 3$, $G-C$ contains $K_{1,3}$ as a subgraph and thus $|G-C|\geq 4$. Also, $G-C$ is connected. We first claim that $G-C$ is not a complete graph. Assume to the contrary that $G-C$ is isomorphic to a complete graph $K_k$ for some positive integer $k$. Since $u$ has an extra neighbor in $C$ that is not adjacent to $v$, $d_G(u)>d_G(v)=\Delta$, a contradiction. Therefore, $G-C$ is isomorphic to neither $P_3$ nor a clique.

Since $C$ is not a clique, by~\Cref{claim:H+G-H} $C$ is isomorphic to $P_3$. Recall that $u$ is a vertex of $N(v)$ adjacent to some vertex of $C$. Note that $H:=G[V(C)\cup \{u\}]$ is a connected induced subgraph of $G$, and $G-H$ is also connected. Since $|H|=4$,  $|G-H|=n-4\geq 7$ and thus $G-H$ cannot be isomorphic to $P_3$. Noting that $H$ is isomorphic to neither a clique (as it contains $P_3$ as an induced subgraph) nor $P_3$, \Cref{claim:H+G-H} implies that $G-H$ must be a clique. Note that $v$ is not connected to $C$ (which is isomorphic to $P_3$) and $vu\in E(G)$. So $V(G)\setminus \{u\}$ admits a partition $\{V(C), V(G-H)\}$ satisfying the conditions in~\Cref{lem:oneP3+onecliques}, and thus $\mathcal{E}_3^{-}(G) \geq n$, a contradiction. 
\end{proof*}

\medskip
In the sequel, by~\Cref{claim:G-N[v]isUnionOfCliques} we may assume that $G - N[v]$ is a disjoint union of cliques and let $V(G)-N[v]=C_1\cup \cdots \cup C_k$ such that $G[C_i]$'s are pairwise disjoint cliques for some $k\geq 1$. Let $N(v)=\{u_1, \ldots, u_{\Delta}\}$. For each $u_i$ with $i\in \{1,\ldots, \Delta\}$, we denote the vertex set of the cliques that are connected to $u_i$ by $C_{i_1}, \ldots, C_{i_t}$ for some $t$ if they exist.

\begin{claim}\label{claim:u+C_i}
Let $u\in N(v)$ be a vertex and let $C_1, \ldots, C_t$ be cliques of $G-N[v]$ such that $u$ is adjacent to some vertex of $C_i$ for each $i\in [t]$.
\begin{enumerate}[label=(\arabic*)]
\setlength{\itemsep}{0em}
    \item\label{1-C_i+u} If there exists $i\in [t]$ such that $|C_i|\geq 3$, then $t=1$ and $G[\{u\}\cup C_1]$ is a clique. 
    \item\label{2-C_i+u} If there exists $i\in [t]$ such that $|C_i|=2$, then $t=1$ and $G[\{u\}\cup C_1]$ is isomorphic to either a clique or $P_3$. 
    \item\label{3-C_i+u} If $|C_i|=1$ for each $i\in[t]$, then $t\leq 2$ and $G[\{u\}\cup \bigcup_{i=1}^t C_i]$ is isomorphic to either a clique or $P_3$.
\end{enumerate}
\end{claim}

\begin{proof*}
Let $u\in N(v)$ be a vertex and let $C_1, \ldots, C_t$ be cliques of $G-N[v]$ such that $v$ is adjacent to some vertex of $C_i$ for each $i\in [t]$. Let $H:=G[\{u\}\cup \bigcup_{i=1}^t C_i]$. We only consider the following conditions: (1) $\exists |C_i|\geq 3$ and $t\geq 1$; (2) $\exists |C_i|=2$ and $t\geq 2$; and (3) $|C_i|=1$ and $t\geq 3$. (In the remaining cases, $H$ is isomorphic to $P_3$, which already satisfies the conclusion of the claim.)
Assume to the contrary that $H$ is not a clique. Under these conditions, $|H|\geq 4$, and thus $H$ is not isomorphic to $P_3$. Moreover, since $v\in V(G)\setminus V(H)$, $G-H$ is connected. 

Note that $G-H$ is not isomorphic to $P_3$. Otherwise, $\Delta=d_G(v)\leq 3$. Since $uv\in E(G)$ and $u$ is adjacent to some vertex of each $C_i$'s for $i\in [t]$, we have that $t\leq 2$. Moreover, as $\Delta\leq 3$, 
$|C_i|\leq 3$ for each $i\in [2]$. However, $|G|\leq |G-H|+|\{u\}|+|C_1|+|C_2|\leq 3+1+3+3=10$, a contradiction. By~\Cref{claim:H+G-H}, $G-H$ must be a clique.

Assume that $G-H$ is a clique. If there is no edge between $G-H$ and $\bigcup_{i=1}^t C_i$, then $G-u$ is a disjoint union of cliques $C_1, \ldots, C_t$, and $G-H$. By~\Cref{lem:UnionOfCliques}, $\mathcal{E}_3^{-}(G) \geq n$, a contradiction. So we know that there exists at least one vertex $u^*\in V(G-H)\setminus \{v\}$ that is adjacent to some vertex, say $w$, in one $C_{i^*}$ for some $i^*\in [t]$. We claim that $u^*$ is not adjacent to $u$ in $G$. Otherwise, if $u^*u\in E(G)$, then $N(v)\subset N(u^*)$ and $w\in N(u^*)\setminus N(v)$, and, hence, $d_G(u^*)>d_G(v)=\Delta$, a contradiction. In this case, since $v$ is not adjacent to any vertex in $C_{i^*}$, $G[C_{i^*}\cup V(G-H)]$ is not a clique. Moreover, as $\Delta\geq 3$, $|G[C_{i^*}\cup V(G-H)]|\geq 4$ and thus $G[C_{i^*}\cup V(G-H)]$ is not isomorphic to $P_3$. Note that $G-G[C_{i^*}\cup V(G-H)]$, which is exactly $H-C_{i^*}$, is connected. By~\Cref{claim:H+G-H}, $H-C_{i^*}$ is isomorphic to either a clique or $P_3$. 

Now we consider the induced subgraph $F:=G[\{v,u,w,u^*\}]$ which is isomorphic to $C_4$. By~\Cref{table:3neg-energy}, $\mathcal{E}_3^{-}(F)=8$. Recall that $u^*\in C_{i^*}$ and $G-H$ is a clique. Now $G-F$ consists of at most four connected components, that is, at most one from $G-H-\{v,u^*\}$, at most one from $C_{i^*}-w$, and at most two from $H-C_{i^*}-u$. Hence, by~\Cref{claim:connected-components}, $\mathcal{E}_3^{-}(G-F)\geq (n-4)-4=n-8$. Altogether, by~\Cref{thm: additivity}, $\mathcal{E}_3^{-}(G)\geq\mathcal{E}_3^{-}(F)+\mathcal{E}_3^{-}(G-F)\geq 8+(n-8)=n$, a contradiction.
\end{proof*}

\begin{claim}\label{claim:Si+Ci}
Let $C_i$ be one clique in $G-N[v]$ and let $S_i=\{u_{i_1}, \ldots, u_{i_\ell}\}$ with $\ell\geq 2$ be the set of vertices each adjacent to some vertex of $C_i$. 
\begin{enumerate}[label=(\arabic*)]
\setlength{\itemsep}{0em}
    \item\label{1-C_i+S_i} If $|C_i|\geq 2$, then $G[C_i\cup S_i]$ is a clique.
    \item\label{2-C_i+S_i} If $|C_i|=1$, then either $G[C_i\cup S_i]$ is a clique or $|S_i|=2$ and $G[C_i\cup \{u_{i_1}, u_{i_2}\}]$ is isomorphic to $P_3$.
\end{enumerate}
\end{claim}

\begin{proof*}
Let $S_i=\{u_{i_1}, \ldots, u_{i_\ell}\}$ with $\ell\geq 2$ and $H:=G[C_i \cup S_i]$. 
If $|C_i|=1$ and $\ell=2$, then $G[C_i\cup \{u_{i_1}, u_{i_2}\}]$ is isomorphic to either $P_3$ or a clique. In the following we consider the case when $|C_i|\geq 2$ or $|C_i|=1$ and $\ell\geq 3$. In this case, $|H|\geq 4$ and thus $H$ is not isomorphic to $P_3$. Assume to the contrary that $H$ is not a clique.

Note that $G-H$ contains the vertex $v$ and it is connected. It follows from~\Cref{claim:H+G-H} that $G-H$ is either a clique or isomorphic to $P_3$. We consider two cases as follows.

\smallskip
\noindent
\textbf{Case (1)}. \emph{Assume that $G-H$ is isomorphic to $P_3$.} 
\smallskip

In this case, we first claim that $|S_i|\geq 3$ (i.e., $\ell\geq 3$). Otherwise, since $|G-H|=3$ and $|S_i|=2$, we have that $\Delta=d_G(v)\leq 4$. Note that by~\Cref{claim:u+C_i}\ref{1-C_i+u}, if $|C_i|\geq 3$, then the vertex $u_{i_1}$ is adjacent to all vertices of $C_i$ and thus $|C_i|+1\leq d_G(u)\leq \Delta=4$, which implies that $|C_i|=3$. So $|C_i|\leq 3$ and $|G|= |C_i|+|S_i|+|G-H|\leq 3+2+3=8$, a contradiction. Thus we have that $\ell\geq 3$. 

Since $S_i$ contains at least $3$ vertices and $G[C_i\cup S_i]$ is not a clique, there exists a vertex $u^*\in S_i$ such that $G[C_i\cup \big(S_i\setminus \{u^*\}\big)]=:H'$ is not a clique as well. Since $G-H$ is isomorphic to $P_3$, $G-H'$ has exactly $4$ vertices and contains $P_3$ as an induced subgraph. Hence, $G-H'$ is neither a clique nor isomorphic to $P_3$.
Also $G-H'$ contains the vertex $v$ (adjacent to $u^*$) and thus is connected. Moreover, $H'$ has $n-4$ ($\geq 7$) vertices and thus is not isomorphic to $P_3$. So $H'$ is neither a clique nor isomorphic to $P_3$, contradicting~\Cref{claim:H+G-H}.

\smallskip
\noindent
\textbf{Case (2)}. \emph{Assume that $G-H$ is a clique.} 
\smallskip

\noindent
\textbf{Case (2.1)}. \emph{Assume that $|C_i|=1$ and $|S_i|\geq 3$.} 
\smallskip

In this case, as $G[C_i\cup S_i]$ is not a clique, $G[S_i]$ is not a clique as well. If $G[S_i]$ has an independent set of size at least $3$, assuming that $\{x,y,z\}$ is an independent set of $G[S_i]$, then $G[C_i\cup\{x,y,z\}]$ is isomorphic to $K_{1,3}$. By~\Cref{table:3neg-energy}, $\mathcal{E}_3^{-}(G[C_i\cup\{x,y,z\}])\approx 5.19 > 5$. Since $G-G[C_i\cup\{x,y,z\}]$ is connected, by~\Cref{claim:connected-components}, $\mathcal{E}_3^{-}(G-G[C_i\cup\{x,y,z\}])\geq (n-4)-1=n-5$. Altogether, by~\Cref{thm: additivity}, $\mathcal{E}_3^{-}(G)\geq\mathcal{E}_3^{-}(G[C_i\cup\{x,y,z\}])+\mathcal{E}_3^{-}(G-G[C_i\cup\{x,y,z\}])>5+(n-5)=n$, a contradiction.  

So assume that the independence number of $G[S_i]$ is at most $2$. Since $G[S_i]$ is not a clique, there exists two non-adjacent vertices $x\in S_i$ and $y\in S_i$ such that $G[C_i\cup \{x,y,v\}]$ is isomorphic to $C_4$. By~\Cref{table:3neg-energy}, $\mathcal{E}_3^{-}(G[C_i\cup \{x,y,v\}])=8$. As $G[S_i]$ has no independent set of size larger than $2$, $G-G[C_i\cup \{x,y,v\}]$ consists of at most two connected components. By~\Cref{claim:connected-components}, $\mathcal{E}_3^{-}(G-G[C_i\cup \{x,y,v\}])\geq (n-4)-2=n-6$. Altogether, by~\Cref{thm: additivity}, $\mathcal{E}_3^{-}(G)\geq\mathcal{E}_3^{-}(G[C_i\cup \{x,y,v\}])+\mathcal{E}_3^{-}(G-G[C_i\cup \{x,y,v\}])\geq 8+n-6>n$, a contradiction.

\smallskip
\noindent
\textbf{Case (2.2)}. \emph{Assume that $|C_i|\geq 2$ and $|S_i|\geq 2$.} 
\smallskip

Recall that $G[C_i\cup S_i]$ is not a clique. We have two possibilities.

Assume that $G[S_i]$ is not a clique. So there are non-adjacent vertices $u_{i_1}\in S_i$ and $u_{i_2}\in S_i$. Since $|C_i|\geq 2$, $G[C_i\cup\{u_{i_1}, u_{i_2}\}]$ has more than $4$ vertices and thus it is isomorphic to neither $P_3$ nor a clique. Furthermore, $G-G[C_i\cup\{u_{i_1}, u_{i_2}\}]$ is connected. We first show that $G-G[C_i\cup\{u_{i_1}, u_{i_2}\}]$ is not isomorphic to $P_3$. Assume not and thus $\Delta=d_G(v)\leq 4$. Note that by~\Cref{claim:u+C_i}\ref{1-C_i+u}, if $|C_i|\geq 3$, then the vertex $u_{i_1}$ is adjacent to all vertices of $C_i$ and thus $|C_i|+1\leq d_G(u)\leq \Delta=4$, which implies that $|C_i|=3$. So $|C_i|\leq 3$ and $|G|= |G[C_i\cup\{u_{i_1}, u_{i_2}\}]|+|G-G[C_i\cup\{u_{i_1}, u_{i_2}\}]|\leq (3+2)+3=8$, a contradiction.  By~\Cref{claim:H+G-H}, $G-G[C_i\cup\{u_{i_1}, u_{i_2}\}]$ must be a clique. If there exist $v'\in C_i$ such that $v'$ is adjacent to both of $u_{i_1}$ and $u_{i_2}$, then $H'_1:=G[\{v,u_{i_1},u_{i_2},v'\}]$ is isomorphic to $C_4$. If no such vertex $v'$ exists, then by~\Cref{claim:u+C_i}\ref{1-C_i+u}, $|C_i|=2$ and $|S_i|=2$ (i.e., $S_i=\{u_{i_1}, u_{i_2}\}$), and, moreover, $H'_2:=G[C_i\cup S_i\cup \{v\}]$ is isomorphic to $C_5$. 
By~\Cref{table:3neg-energy}, for each $i\in \{1,2\}$, $\mathcal{E}_3^{-}(H'_i)\geq \min\{8,~8.472 \}\geq 8$. 
Noting that each $G-H'_i$ consists of at most three connected components (since $G-G[C_i\cup S_i]$ is a clique) and $|H'_i|\in \{4,5\}$, by~\Cref{claim:connected-components}, $\mathcal{E}_3^{-}(G-H'_i)\geq (n-5)-3=n-8$. Altogether, by~\Cref{thm: additivity}, $\mathcal{E}_3^{-}(G)\geq\mathcal{E}_3^{-}(H'_i)+\mathcal{E}_3^{-}(G-H'_i)\geq 8+n-8=n$, a contradiction.

Assume that $G[S_i]$ is a clique. In this case, by~\Cref{claim:u+C_i}\ref{1-C_i+u} and \ref{2-C_i+u}, $|C_i|=2$ and there exists at least one vertex $u\in S_i$ such that $G[C_i\cup \{u\}]$ is isomorphic to $P_3$. Assume without loss of generality that $C_i=\{w_1, w_2\}$ and $uw_1w_2$ is an induced $P_3$. Let $u'\in S_i\setminus \{u\}$ be an arbitrary vertex. Note that $H'':=G[\{u,u',w_1,w_2\}]$ has $4$ vertices and thus is not a clique. Since $G-H''$ is connected (as it contains $v$), by~\Cref{claim:H+G-H}, $G-H''$ is isomorphic to either $P_3$ or a clique. If $G-H''$ is isomorphic to $P_3$, then $|G|=|H''|+|G-H''|=4+3=7$, a contradiction. Now we assume that $G-H''$ is a clique. In this case, $H''':=G[\{v,u,u',w_1,w_2\}]$ is isomorphic to one of the following graphs: $H_7$, $H_8$, and $H_{10}$, as shown in~\Cref{fig:v-u-C_i}.
\begin{figure}[htbp]
\centering
\begin{subfigure}[t]{.32\textwidth}
\centering
\begin{tikzpicture}[every node/.style={draw,circle,fill=black,inner sep=1.5pt}]
\node (v) at (1,2) {};\node[draw=none, fill=none, right=2pt] at (v) {$v$};

\node (1) at (0,0) {}; \node (2) at (1,0) {}; \node (3) at (1,1) {}; \node (4) at (0,1) {};
\node[draw=none, fill=none, right=2pt] at (3) {$u'$};
\node[draw=none, fill=none, left=2pt] at (4) {$u$};
\node[draw=none, fill=none, left=2pt] at (1) {$w_1$};
\node[draw=none, fill=none, right=2pt] at (2) {$w_2$};

\draw (1)--(2)--(3)--(4)--(1);
\draw (3)--(v);
\draw (4)--(v);
\end{tikzpicture}
\caption{$H_7$}
\label{fig:type1}
\end{subfigure}%
\begin{subfigure}[t]{.32\textwidth}
\centering
\begin{tikzpicture}[every node/.style={draw,circle,fill=black,inner sep=1.5pt}]
\node (v) at (1,2) {};\node[draw=none, fill=none, right=2pt] at (v) {$v$};

\node (1) at (0,0) {}; \node (2) at (1,0) {}; \node (3) at (1,1) {}; \node (4) at (0,1) {};
\node[draw=none, fill=none, right=2pt] at (3) {$u'$};
\node[draw=none, fill=none, left=2pt] at (4) {$u$};
\node[draw=none, fill=none, left=2pt] at (1) {$w_1$};
\node[draw=none, fill=none, right=2pt] at (2) {$w_2$};

\draw (4)--(1)--(2);
\draw (3)--(v);
\draw (4)--(v);
\draw (1)--(3)--(4);
\end{tikzpicture}
\caption{$H_{8}$}
\label{fig:type2}
\end{subfigure}
\begin{subfigure}[t]{.32\textwidth}
\centering
\begin{tikzpicture}[every node/.style={draw,circle,fill=black,inner sep=1.5pt}]
\node (v) at (1,2) {};
\node[draw=none, fill=none, right=2pt] at (v) {$v$};
\node (1) at (0,0) {}; 
\node (2) at (1,0) {}; 
\node (3) at (1,1) {}; 
\node (4) at (0,1) {};
\node[draw=none, fill=none, right=2pt] at (3) {$u'$};
\node[draw=none, fill=none, left=2pt] at (4) {$u$};
\node[draw=none, fill=none, left=2pt] at (1) {$w_1$};
\node[draw=none, fill=none, right=2pt] at (2) {$w_2$};

\draw (4)--(1)--(2)--(3);
\draw (3)--(v);
\draw (4)--(v);
\draw (1)--(3)--(4);
\end{tikzpicture}
\caption{$H_{10}$}
\label{fig:type3}
\end{subfigure}
\caption{The possibilities of $H'''$ in Case (2.2) of \Cref{claim:Si+Ci}}
\label{fig:v-u-C_i}
\end{figure}
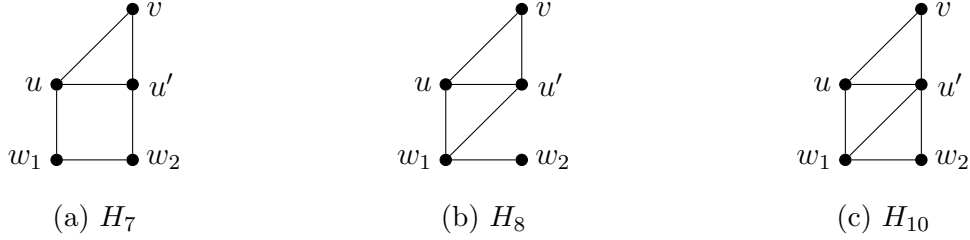
By~\Cref{table:3neg-energy}, $\mathcal{E}_3^{-}(H''')\geq \min\{9.602,~6.804,~7.530\}\geq 6$. Since $G-H''$ is a clique, $G-H'''$ is also a clique and thus by~\Cref{claim:connected-components}, $\mathcal{E}_3^{-}(G-H''')\geq (n-5)-1=n-6$. By~\Cref{thm: additivity}, $\mathcal{E}_3^{-}(G)\geq\mathcal{E}_3^{-}(H''')+\mathcal{E}_3^{-}(G-H''')\geq 6+n-6=n$, a contradiction.

\medskip
We complete the proof of the claim.
\end{proof*}

\medskip
For each clique $C_i$, let $S_i$ denote the set of vertices in $N(v)$ that are adjacent to some vertex in $C_i$.
The next claim concerns the vertices in $N(v)\setminus \bigcup_{i}S_i$, that is, those vertices in $N(v)$ that are not adjacent to any vertex in any $C_i$.

\begin{claim}\label{claim:non-S_i-vertices}
Let $G'=G-v$. Let $X=\{u'_1, u'_2, \ldots, u'_s\}$ denote a subset of $N(v)\setminus \bigcup_{i}S_i$ such that $G'[X]$ is connected. Then $G'[X]$ is either a clique or isomorphic to $P_3$.
\end{claim}

\begin{proof*}
Let $G'=G-v$ and let $X=\{u'_1, u'_2, \ldots, u'_s\}$ denote a subset of $N(v)\setminus \bigcup_{i}S_i$ such that $G'[X]$ is connected. Assume to the contrary that $G'[X]$ is neither a clique nor isomorphic to $P_3$. Note that $|X|\geq 4$ and $G'[X]$ must contain an induced path, say $u'_1u'_2u'_3$, isomorphic to $P_3$. Since $|X|\geq 4$, there exists a vertex $u_4'\in X$ such that $H:=G[\{u'_1,u'_2,u'_3,u_4'\}]$ is connected. So the induced subgraph $H$ is neither a clique nor isomorphic to $P_3$. We first claim that $G-H$ is not a clique. Otherwise, $v$ is adjacent to all vertices in $G-H$ and also vertices in $H$, thus $v$ is a dominating vertex, a contradiction. 
Furthermore, as $n\geq 11$, $|G-H|\geq 7$ and thus $G-H$ is not isomorphic to $P_3$ as well. Since every pair of vertices in $X\setminus V(H)$ can be connected via a path through the vertex $v$, $G-H$ is connected and it contradicts~\Cref{claim:H+G-H}. 
\end{proof*}

The above claim implies that if $|X|\geq 4$, then $G'[X]$ is a clique.

\medskip
Based on the three exceptional cases in~\Cref{claim:u+C_i}\ref{2-C_i+u} and \ref{3-C_i+u}, and \Cref{claim:Si+Ci}\ref{2-C_i+S_i}, we define the following three types of vertices in $N(v)$, as shown in~\Cref{fig:G-v}. Let $u\in N(v)$ be a vertex.

\begin{figure}[ht]
\centering
\begin{tikzpicture}[
  every node/.style={draw,circle,fill=black,inner sep=1.5pt},
  xscale=1.5
]

\node (v) at (4.75,4) {};
\node[draw=none, fill=none, above=2pt] at (v) {$v$};

\node (u1)  at (1,2.5) {};
\node (u2)  at (2,2.5) {};
\node (u4)  at (2.75,2.5) {};
\node (u3)  at (4,2.5) {};
\node (u5)  at (5,2.5) {};
\node (u6)   at (5.5,2.5) {};
\node (u7)   at (6,2.5) {};
\node (u8) at (6.5,2.5) {};
\node (u9)  at (7,2.5) {};
\node (u10)  at (7.5,2.5) {};
\node (u11)  at (8,2.5) {};
\node (u12)  at (9,2.5) {};

\draw (v)--(u1);
\draw (v)--(u2);
\draw (v)--(u3);
\draw (v)--(u4);
\draw (v)--(u5);
\draw (v)--(u6);
\draw (v)--(u7);
\draw (v)--(u8);
\draw (v)--(u9);
\draw (v)--(u10);
\draw (v)--(u11);
\draw (v)--(u12);

\node (w11) at (1,1.5) {}; 
\node (w12) at (1,0.5) {};
\node (w21) at (1.8,1.5) {}; 
\node (w22) at (2.2,1.5) {};
\node (w31) at (4,1.5) {}; 
\node (w32) at (4.5,2.5) {};
\node (w41) at (3.25,2.5)  {}; 
\node (w42) at (3,1.5) {};
\node (w81) at (6.25,1.5) {}; 
\node (w82) at (6.75,1.5) {};
\node (w91) at (7,1.5) {}; 
\node (w102) at (7.5,1.5) {};
\node (w111) at (8.5,2.5) {};

\draw (v)--(w32);
\draw (v)--(w41);
\draw (u1)--(w11);  
\draw (w11)--(w12);
\draw (u2)--(w21);  
\draw (u2)--(w22);
\draw (u3)--(w31);  
\draw (u3)--(w32);
\draw (w42)--(w41);  
\draw (u4)--(w42);
\draw (u5)--(u6)--(u7);  
\draw (u8)--(w81);  
\draw (u8)--(w82);
\draw (u9)--(w91); 
\draw (w81)--(w82);
\draw (w91)--(u10)--(u9);
\draw (u10)--(w102)--(u9);
\draw (w91)--(w102);
\draw (u11)--(w111);
\draw (v)--(w111);

\node[draw=none, fill=none, below=5pt] at (1,0.5) {\footnotesize type $1$};
\node[draw=none, fill=none, below=5pt] at (2,0.5) {\footnotesize type $2$};
\node[draw=none, fill=none, below=5pt] at (3,0.5) {\footnotesize type $3$};
\end{tikzpicture}
\caption{Possibilities of $G-v$}
\label{fig:G-v}
\end{figure}
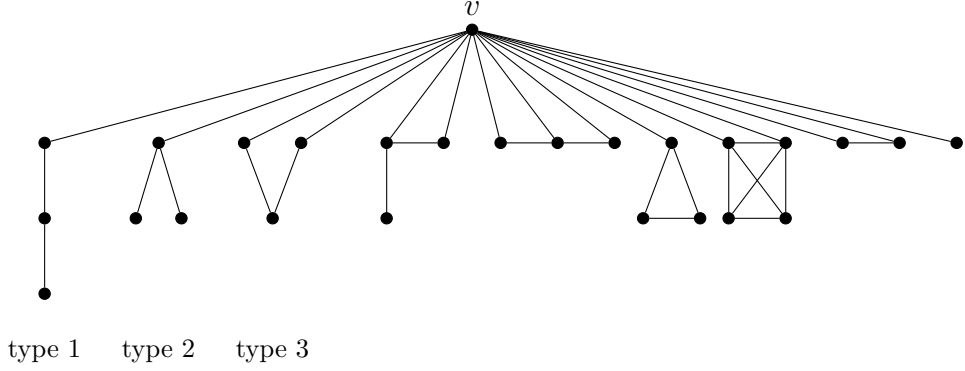

\begin{enumerate}[label=(\arabic*)]
\setlength{\itemsep}{0em}
    \item The vertex $u$ is of \emph{type $1$} if $u$ is adjacent to some vertex of one $C_i$ with $|C_i|=2$ such that $G[\{u\}\cup C_i]$ is isomorphic to $P_3$;
    \item The vertex $u$ is of \emph{type $2$} if $u$ is adjacent to some vertex of $C_i$ and $C_j$ with $|C_i|=|C_j|=1$ such that $G[\{u\}\cup C_i\cup C_j]$ is isomorphic to $P_3$;
    \item The vertex $u$ is of \emph{type $3$} if $u$ is adjacent to some vertex of one $C_i$ with $|C_i|=1$ and there is another $u'\in N(v)$ adjacent to $C_i$ such that $G[\{u,u'\}\cup C_i]$ is isomorphic to $P_3$. 
\end{enumerate}

Note that \Cref{claim:Si+Ci} implies that $u\in S_i$ for some $i$ in some clique of $G-v$ must be in a unique clique unless $u$ is of type $2$. The next claim tells us that any vertex $u\in N(v)$ of type $i$ for $i\in [3]$ cannot be of type $j$ for some $j\in [3]$, and must be in a unique copy of $P_3$. 

\begin{claim}\label{claim:uniqueP3}
Let $u\in N(v)$ be a vertex of type $i$ for $i\in [3]$. Then $u$ cannot be a vertex of type $j$ for any distinct $j\in [3]$ and $u$ is in a unique path isomorphic to $P_3$ in $G-v$.  
\end{claim}

\begin{proof*}
Let $u\in N(v)$ be a vertex. 

If $u$ is of type 1, then $u$ can only connected to one $C_i$. Otherwise, $u$ is connected to both $C_i$ and $C_j$ where $G[\{u\}\cup C_i]$ is isomorphic to $P_3$ (with $|C_i|=2$) and $|C_j|\geq 1$. But it contradicts~\Cref{claim:u+C_i}\ref{2-C_i+u}.

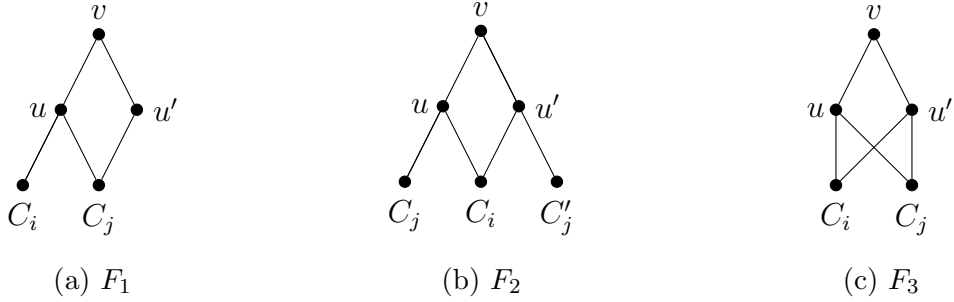
\begin{figure}[ht]
\centering
\begin{subfigure}[t]{.32\textwidth}
\centering
\begin{tikzpicture}[every node/.style={draw,circle,fill=black,inner sep=1.5pt}]
\node (v) at (0, 2) {};  
\node[draw=none, fill=none, above=2pt] at (v) {$v$};
\node (1) at (-1, 0) {};
\node (b) at (0, 0) {};
\node (3) at (-0.5, 1) {};
\node[draw=none, fill=none, left=2pt] at (3) {$u$};
\node (4) at (0.5, 1) {};
\draw (v) -- (1);
\draw (4) -- (v);
\draw (3) -- (b);
\draw (1) -- (3);
\draw (4) -- (b);
\node[draw=none, fill=none, right=2pt] at (4) {$u'$};
\node[draw=none, fill=none, below=2pt] at (1) {$C_i$};
\node[draw=none, fill=none, below=2pt] at (b) {$C_j$};
\end{tikzpicture}
\caption{$F_1$}
\label{fig:F1}
\end{subfigure}%
\begin{subfigure}[t]{.32\textwidth}
\centering
\begin{tikzpicture}[every node/.style={draw,circle,fill=black,inner sep=1.5pt}]
\node (v) at (0, 2) {};  
\node[draw=none, fill=none, above=2pt] at (v) {$v$};
\node (1) at (-1, 0) {};
\node (b) at (0, 0) {};
\node (3) at (-0.5, 1) {};
\node[draw=none, fill=none, left=2pt] at (3) {$u$};
\node (4) at (0.5, 1) {};
\node (c) at (1, 0) {};
\draw (v) -- (1);
\draw (4) -- (v);
\draw (v) -- (c);
\draw (3) -- (b);
\draw (1) -- (3);
\draw (4) -- (b);
\node[draw=none, fill=none, right=2pt] at (4) {$u'$};
\node[draw=none, fill=none, below=2pt] at (c) {$C_j'$};
\node[draw=none, fill=none, below=2pt] at (1) {$C_j$};
\node[draw=none, fill=none, below=2pt] at (b) {$C_i$};
\end{tikzpicture}
\caption{$F_2$}
\label{fig:F2}
\end{subfigure}
\begin{subfigure}[t]{.32\textwidth}
\centering
\begin{tikzpicture}[every node/.style={draw,circle,fill=black,inner sep=1.5pt}]
\node (v) at (0, 2) {};  
\node[draw=none, fill=none, above=2pt] at (v) {$v$};
\node (1) at (-0.5, 0) {};
\node (b) at (0.5, 0) {};
\node (3) at (-0.5, 1) {};
\node[draw=none, fill=none, left=1pt] at (3) {$u$};
\node (4) at (0.5, 1) {};
\draw (v) -- (3);
\draw (4) -- (v);
\draw (3) -- (b);
\draw (4)--(1) -- (3);
\draw (4) -- (b);
\node[draw=none, fill=none, right=2pt] at (4) {$u'$};
\node[draw=none, fill=none, below=2pt] at (1) {$C_i$};
\node[draw=none, fill=none, below=2pt] at (b) {$C_j$};
\end{tikzpicture}
\caption{$F_3$}
\label{fig:F3}
\end{subfigure}
\caption{Configurations $F_1, F_2$, and $F_3$}
\label{fig:F123}
\end{figure}

If $u$ is of type 2, then by~\Cref{claim:u+C_i}\ref{3-C_i+u} $u$ can only connected to $C_i$ and $C_j$ with $|C_i|=|C_j|=1$. Assume to the contrary that $u$ is also of type 3, i.e., $C_j$ is adjacent to another $u'\in N(v)$. Then $G$ contains $F_1$ as an induced subgraph, as shown in~\Cref{fig:F1}. 

If $u$ is of type 3, then by symmetry $u'$ is also of type 3. By~\Cref{claim:u+C_i}\ref{3-C_i+u} each of $u$ and $u'$ can be only connected to at most two cliques of size $1$. So $d_{G''}(u)\leq 2$ and $d_{G''}(u')\leq 2$ where $G''=G-N[v]$. Assume to the contrary that (1) there is $C_j$ with $j\neq i$ and $|C_j|=1$ such that $G[\{u, u'\}\cup C_j]$ is isomorphic to $P_3$; or (2) there are two $C_j$ and $C_j'$ with $|C_j|=|C_j'|=1$ such that each of $G[\{u\}\cup C_i\cup C_j]$ and $G[\{u'\}\cup C_i\cup C_j']$ is isomorphic to $P_3$ (i.e., $u$ and $u'$ are also of type $2$). In case (1), $G$ contains $F_2$ as an induced subgraph; and in case (2) $G$ contains $F_3$ as an induced subgraph. See~\Cref{fig:F2,fig:F3}.

For each $F_\ell$ with $\ell\in[3]$, it contains $C_4$ as an induced subgraph, that is to say, $G$ contains an induced subgraph $H$ isomorphic to $C_4$. Moreover, $G-H$ consists of at most three connected components. By~\Cref{table:3neg-energy}, $\mathcal{E}_3^{-}(H)=8$, and by~\Cref{claim:connected-components}, $\mathcal{E}_3^{-}(G-H)\geq (n-4)-3=n-7$. Altogether, by~\Cref{thm: additivity}, $\mathcal{E}_3^{-}(G)\geq\mathcal{E}_3^{-}(H)+\mathcal{E}_3^{-}(G-H)\geq 8+n-7>n$, a contradiction. 
\end{proof*}

Hence, for each vertex $u\in N(v)$ of type $i$ with $i\in [3]$, the corresponding path $G[\{u\}\cup C_i]$, $G[\{u\}\cup C_i\cup C_j]$, or $G[\{u,u'\}\cup C_i]$ is unique. 

\begin{claim}\label{claim:NonAdjacentP3}
Let $k\in [3]$ be an integer. Each vertex $u\in N(v)$ of type $k$ is not adjacent to any other $w\in N(v)$.
\end{claim}

\begin{proof*}
Let $u$ be a vertex of type $k$ for $k\in [3]$. Assume to the contrary that there is $w\in N(v)$ such that $uw\in E(G)$. Let $Q_w$ denote the union of all cliques $C$ of $G-N[v]$ such that $w$ is adjacent to some vertex in $C$. Note that $Q_w$ might be an empty set. Let $$
    Q_u =
        \begin{cases}
            G[\{u\}\cup C_i], & \text{if $u$ is of type 1};\\
            G[\{u\}\cup C_i\cup C_j], & \text{if $u$ is of type 2};\\
            G[\{u,u'\}\cup C_i], & \text{if $u$ is of type 3}.
        \end{cases}
    $$ So $Q_u$ is isomorphic to $P_3$. Let $H=G[Q_u\cup Q_w \cup \{w\}]$. Note that $H$ is connected, which is neither a clique nor isomorphic to $P_3$. Moreover, $G-H$ is also connected, as it contains the vertex $v$. It follows from~\Cref{claim:H+G-H} that $G-H$ is either isomorphic to $P_3$ or a clique. 

    We first claim that $G-H$ is not isomorphic to $P_3$. Assume not. Since $v\in V(G-H)$, $\Delta=d_G(v)\leq 4$. If $|Q_w|\geq 3$, then by~\Cref{claim:u+C_i}\ref{1-C_i+u} the vertex $w$ is adjacent to all the vertices of $Q_w$, which contradicts the fact that $|Q_w|+2\leq d_G(w)\leq \Delta=4 $; Otherwise, $|Q_w|\leq 2$, and so $|G|=|H|+|G-H|\leq 6+3=9$, a contradiction.  

    In what follows, we assume that $G-H$ is a clique. We consider three cases based on the types of $u$.

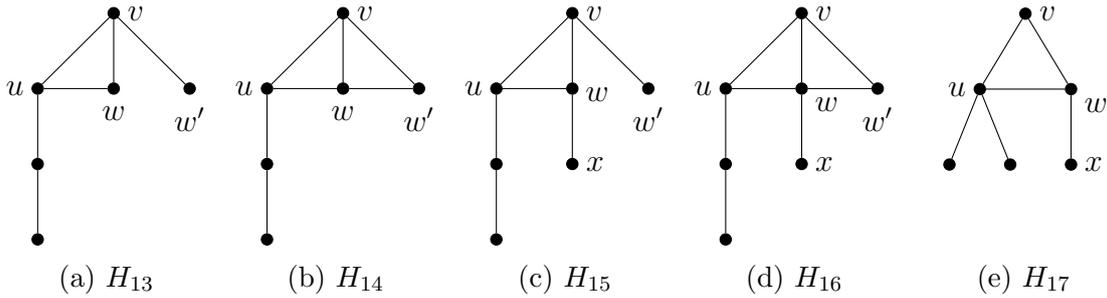
\begin{figure}[htbp]
\centering 
\begin{subfigure}[t]{.18\textwidth}
\centering  
\begin{tikzpicture}[every node/.style={draw,circle,fill=black,inner sep=1.5pt}]
\node (v)  at (0,2) {};
\node (a)  at (-1,1) {};
\node (b)  at (1,1) {};
\node (c)  at (0,1) {};
\node (d1) at (-1,0) {};
\node (d2) at (-1,-1) {};
\draw (v)--(a)--(c)--(v);
\draw (v)--(b);
\draw (a)--(d1)--(d2);
\node[draw=none, fill=none, right=2pt] at (v) {$v$};
\node[draw=none, fill=none, left=2pt] at (a) {$u$};
\node[draw=none, fill=none, below=2pt] at (c) {$w$};
\node[draw=none, fill=none, below=2pt] at (b) {$w'$};
\end{tikzpicture}
\caption{$H_{13}$}
\label{fig:newH13} 
\end{subfigure}
\begin{subfigure}[t]{.18\textwidth}
\centering  
\begin{tikzpicture}[every node/.style={draw,circle,fill=black,inner sep=1.5pt}]
\node (v)  at (0,2) {};
\node (a)  at (-1,1) {};
\node (b)  at (1,1) {};
\node (c)  at (0,1) {};
\node (d1) at (-1,0) {};
\node (d2) at (-1,-1) {};
\draw (v)--(a)--(c)--(v);
\draw (v)--(b)--(c);
\draw (a)--(d1)--(d2);
\node[draw=none, fill=none, right=2pt] at (v) {$v$};
\node[draw=none, fill=none, left=2pt] at (a) {$u$};
\node[draw=none, fill=none, below=2pt] at (c) {$w$};
\node[draw=none, fill=none, below=2pt] at (b) {$w'$};
\end{tikzpicture}
\caption{$H_{14}$}
\label{fig:newH14} 
\end{subfigure}
\begin{subfigure}[t]{.18\textwidth}
\centering  
\begin{tikzpicture}[every node/.style={draw,circle,fill=black,inner sep=1.5pt}]
\node (v)  at (0,2) {};
\node (a)  at (-1,1) {};
\node (b)  at (0,1) {};
\node (c)  at (1,1) {};
\node (d1) at (-1,0) {};
\node (d2) at (-1,-1) {};
\node (e)  at (0,0) {};
\draw (v)--(a)--(b)--(v);
\draw (v)--(c);
\draw (a)--(d1)--(d2);
\draw (b)--(e);
\node[draw=none, fill=none, right=2pt] at (v) {$v$};
\node[draw=none, fill=none, left=2pt] at (a) {$u$};
\node[draw=none, fill=none, below=2pt,right=2pt] at (b) {$w$};
\node[draw=none, fill=none, below=2pt] at (c) {$w'$};
\node[draw=none, fill=none, right=2pt] at (e) {$x$};
\end{tikzpicture}
\caption{$H_{15}$}
\label{fig:H16} 
\end{subfigure}
\begin{subfigure}[t]{.18\textwidth}
\centering  
\begin{tikzpicture}[every node/.style={draw,circle,fill=black,inner sep=1.5pt}]
\node (v)  at (0,2) {};
\node (a)  at (-1,1) {};
\node (b)  at (1,1) {};
\node (c)  at (0,1) {};
\node (d1) at (-1,0) {};
\node (d2) at (-1,-1) {};
\node (e)  at (0,0) {};
\draw (v)--(a)--(c)--(v);
\draw (v)--(b)--(c);
\draw (a)--(d1)--(d2);
\draw (c)--(e);
\node[draw=none, fill=none, right=2pt] at (v) {$v$};
\node[draw=none, fill=none, left=2pt] at (a) {$u$};
\node[draw=none, fill=none, below=6pt, right=2pt] at (c) {$w$};
\node[draw=none, fill=none, below=2pt] at (b) {$w'$};
\node[draw=none, fill=none, right=2pt] at (e) {$x$};
\end{tikzpicture}
\caption{$H_{16}$}
\label{fig:H17} 
\end{subfigure}
\begin{subfigure}[t]{.18\textwidth}
\centering
\begin{tikzpicture}[every node/.style={draw,circle,fill=black,inner sep=1.5pt}]
\node (v)  at (0,2) {};
\node (a)  at (-0.6,1) {};
\node (b)  at (0.6,1) {};
\node (a1) at (-1,0) {};
\node (a2) at (-0.2,0) {};
\node (b1) at (0.6,0) {};
\draw (v)--(a);
\draw (v)--(b);
\draw (a)--(b);
\draw (a2)--(a)--(a1);
\draw (b)--(b1);
\node[draw=none, fill=none, right=2pt] at (v) {$v$};
\node[draw=none, fill=none, left=2pt] at (a) {$u$};
\node[draw=none, fill=none, below=6pt,right=2pt] at (b) {$w$};
\node[draw=none, fill=none, right=2pt] at (b1) {$x$};
\node[draw=none, fill=none, left=2pt] at (a1) {};
\node[draw=none, fill=none, right=2pt] at (a2) {};
\node[draw=none, fill=none, left=2pt] at (-1,-1) {};
\end{tikzpicture}
\caption{$H_{17}$}
\label{fig:H18}
\end{subfigure}
\caption{The configurations in~\Cref{claim:NonAdjacentP3}}
\label{fig:uNOu'}
\end{figure}

\noindent
(1) \emph{Assume that $u$ is of type 1}. Note that $|G-H|\geq 2$, as otherwise $\Delta\leq 3$, then $|Q_w|\leq 1$ and $|G|\leq|Q_u|+|Q_w|+|\{w\}|+2\leq 7$, a contradiction. Moreover, there exists a vertex $w'\in N(v)$ not adjacent to $u$, as otherwise we have $d_G(u)=d_G(v)+1>\Delta$, a contradiction. In this case, $V(G-H)\subset N[v]$.

If $|Q_w|=0$, then we consider the induced subgraph $H':=G[Q_u\cup \{v,w,w'\}]$. Note that $H'$ is isomorphic to $H_{13}$ if $w$ is not adjacent to $w'$, and is isomorphic to $H_{14}$ otherwise. See~\Cref{fig:newH13} and \Cref{fig:newH14}. By~\Cref{table:3neg-energy}, $\mathcal{E}_3^{-}(H')\geq \min\{8.068, 8.499\}>7$. Note that $G-H'$ is connected. By~\Cref{thm: additivity}, $\mathcal{E}_3^{-}(G)\geq\mathcal{E}_3^{-}(H')+\mathcal{E}_3^{-}(G-H')> 7+(n-6-1)=n$, a contradiction.

If $|Q_w|\geq 1$, then let $x\in V(Q_w)\cap N(w)$. We consider the induced subgraph $H':=G[Q_u\cup \{v,w,w',x\}]$. Note that $H'$ is isomorphic to $H_{15}$ if $w$ is not adjacent to $w'$, and is isomorphic to $H_{16}$ otherwise. See~\Cref{fig:H16} and \Cref{fig:H17}. 
By~\Cref{table:3neg-energy}, $\mathcal{E}_3^{-}(H')\geq \min \{10.679,~11.074\}>9$. Moreover, since $G-H$ is a clique, $G-H'$ has at most two connected components, and thus by~\Cref{claim:connected-components}, $\mathcal{E}_3^{-}(G-H')\geq (n-7)-2=n-9$. By~\Cref{thm: additivity}, $\mathcal{E}_3^{-}(G)\geq\mathcal{E}_3^{-}(H')+\mathcal{E}_3^{-}(G-H')> 9+(n-9)=n$, a contradiction.  

\medskip
\noindent
(2) \emph{Assume that $u$ is of type 2}.

If $|Q_w|=0$, then $G$ contains $H':=G[\{w\}\cup Q_u]$ as an induced subgraph, which is isomorphic to $K_{1,3}$; Moreover, $G-H'$ is connected.  By~\Cref{table:3neg-energy}, $\mathcal{E}_3^{-}(H')\geq 5$. By~\Cref{thm: additivity}, $\mathcal{E}_3^{-}(G)\geq\mathcal{E}_3^{-}(H')+\mathcal{E}_3^{-}(G-H')> 5+(n-4-1)=n$, a contradiction.
       
If $|Q_w|\geq 1$, then let $x\in V(Q_w)\cap N(w)$. So $H':=G[\{v,w,x\}\cup Q_u]$ is an induced subgraph of $G$, which is isomorphic to $H_{17}$ (shown in~\Cref{fig:H18}). By~\Cref{table:3neg-energy}, $\mathcal{E}_3^{-}(H')\approx 9.136\geq 9$. Moreover, since $G-H$ is a clique, $G-H'$ has at most two connected components, and thus by~\Cref{claim:connected-components}, $\mathcal{E}_3^{-}(G-H')\geq (n-6)-2=n-8$. By~\Cref{thm: additivity}, $\mathcal{E}_3^{-}(G)\geq\mathcal{E}_3^{-}(H')+\mathcal{E}_3^{-}(G-H')\geq 9+n-8=n+1$, a contradiction. 

\medskip
\noindent
(3) \emph{Assume that $u$ is of type 3}. So $G$ contains $H':=G[\{v\}\cup Q_u]$ as an induced subgraph, which is isomorphic to $C_4$. Moreover, since $G-H$ is a clique, $G-H'$ has at most two connected components. Hence, by~\Cref{thm: additivity}, $\mathcal{E}_3^{-}(G)\geq\mathcal{E}_3^{-}(H')+\mathcal{E}_3^{-}(G-H')> 8+(n-4-2)>n$, a contradiction. 
\end{proof*}

\medskip
It follows from \Cref{claim:uniqueP3} and \Cref{claim:NonAdjacentP3} that in $G-v$, any two $P_3$'s each containing a vertex of some type $k \in [3]$ are vertex-disjoint; Moreover, these $P_3$'s are not adjacent to any other connected components in $G-v$.

The next claim demonstrates that for any vertex $u \in N(v)$, if $u\in S_i$ and $Q_u=G[S_i\cup C_i]$ is a clique, then $u$ is not adjacent to every other vertex $u'\in N(v)\setminus S_i$.

\begin{claim}\label{claim:last claim}
Assume that $K=G[S_i\cup C_i]$ is a clique in $G-v$ with $S_i=\{u_{i_1}, \ldots, u_{i_t}\}$.  Then $u_{i_j}$ is not adjacent to any of $u'\in N(v)\setminus S_i$, unless $|C_i|=1$, $|S_i|=1$, and $u'$ is an isolated vertex in $G-(V(K)\cup \{v\})$.  
\end{claim}

\begin{proof*}
Assume to the contrary that $u\in S_i$ is adjacent to some other $u'\in N(v)\setminus S_i$ and the exceptional case does not hold. By~\Cref{claim:NonAdjacentP3}, $u'$ can not be a vertex of type $i$ for $i\in [3]$. Thus we consider the following two possibilities.

\smallskip
\noindent
{\bf Case 1}. \emph{The vertex $u'$ is adjacent to some vertex of some clique $C_j$ ($j\neq i$) in $G-N[v]$.}
\smallskip

Let $F=G[S_j\cup C_j]$ such that $u'\in S_j$. Note that by~\Cref{claim:Si+Ci} and \Cref{claim:NonAdjacentP3} $F$ is also a clique in $G-v$. Recalling that $u\in S_i$ and $u'\in S_j$, 
let $a \in C_i$ and $b \in C_j$ denote the neighbors of $u$ and $u'$, respectively. Then $ auu'b $ is an induced path of length $4$ of $G$. We consider the induced connected subgraph $H:=G[\ C_i\cup C_j\cup \{u,u'\}]$. Since $H$ has at least four vertices and contains $auu'b$ as an induced subgraph, it is neither a clique nor isomorphic to $P_3$. Moreover, since $G-H$ contains the vertex $v$, it is connected. By~\Cref{claim:H+G-H}, $G-H$ is isomorphic to either $P_3$ or a clique.

If $G-H$ is isomorphic to $P_3$, then $d_G(v)=\Delta \leq 4$. Since $u$ is adjacent to all vertices of $C_i$ and $u'$ is adjacent to all vertices of $C_j$, $|C_i|\leq 2$ and $|C_j|\leq 2$. Hence, $|G|=|G-H|+|C_i|+|C_j|+2\leq 3+2+2+2\leq 9$, a contradiction.

If $G-H$ is a clique, then $V(G-H)=N[v]\setminus \{u,u'\}$. Recall that by~\Cref{claim:u+C_i}\ref{1-C_i+u} and~\Cref{claim:NonAdjacentP3}, each of $G[\{u\}\cup C_i]$ and $G[\{u'\}\cup C_j]$ is a clique. Since $G-N[v]$ is a disjoint union of cliques $G[C_1], \ldots, G[C_k]$ and there is no edge connecting any vertex of $C_i$ and $C_j$ for any distinct $i,j\in [k]$, $G[\{u\}\cup C_i]$ and $G[\{u'\}\cup C_j]$ are pairwise vertex-disjoint. We consider two subcases:
\begin{itemize}
    \item Assume that there exists $u''\in G-H$ that adjacent to some of $C_i$ or $C_j$, say $C_j$. 
    
    By~\Cref{claim:Si+Ci}\ref{1-C_i+S_i}, $u''$ is adjacent to all the vertices of $C_j$. In this case, $u''$ is not adjacent to any vertex of $C_i$ (in particular, $u\in C_i$), as otherwise by~\Cref{claim:Si+Ci} $G[\{u,u',u''\}\cup C_i\cup C_j]$ is a clique which is not possible. Moreover, $u''$ must be adjacent to $u'$, as otherwise, by~\Cref{claim:Si+Ci}, it must hold that $|C_j|=1$ and $u'$ is of type $3$ and adjacent to $u$, which contradicts~\Cref{claim:NonAdjacentP3}. Noting that $G-H$ is a clique, $d_G(v)=|G\setminus \{v\}|-|C_i|-|C_j|$, and $d_G(u'')\geq |G\setminus \{u'',u\}|-|C_i|$, we have that $\Delta=d_G(v)\geq d_G(u'')$, which implies that $|C_j|\leq 1$. Hence, $C_j=\{b\}$.

    Then if $|C_i|=1$, then $H':=G[\{v,u,u',a,b\}]$ is an induced subgraph of $G$ which is isomorphic to $H_4$. If $|C_i|\geq 2$, then $H'':=G[\{v,u,u',a, a',b\}]$ is an induced subgraph of $G$ which is isomorphic to $H_5$, where $a'$ is a vertex of $C_i\setminus \{a\}$. See~\Cref{fig:H4H5inCase1inClaim3.9}. 
\begin{figure}[ht]
\centering
\begin{subfigure}[t]{.45\textwidth}
\centering
\begin{tikzpicture}[every node/.style={draw,circle,fill=black,inner sep=1.5pt}]
\begin{scope}[xshift=0cm]
\node (v) at (-0.5, 2) {};  
\node[draw=none, fill=none, above=2pt] at (v) {$v$};
\node (1) at (-1, 1) {};
\node[draw=none, fill=none, left=2pt] at (1) {$u$};
\node (b) at (0, 1) {};
\node (3) at (-1, 0) {};
\node[draw=none, fill=none, left=2pt] at (3) {$a$};
\node (c) at (0, 0) {};
\draw (v) -- (1);
\draw (v) -- (b);
\draw (b) -- (c);
\draw (1) -- (b);
\draw (1) -- (3);
\node[draw=none, fill=none, right=2pt] at (b) {$u'$};
\node[draw=none, fill=none, right=2pt] at (c) {$b$};
\end{scope}
\end{tikzpicture}
\caption{$G[\{v,u,u',a,b\}]\cong H_4$}
\label{fig:H4inCase1}
\end{subfigure}
\begin{subfigure}[t]{.45\textwidth}
\centering
\begin{tikzpicture}[every node/.style={draw,circle,fill=black,inner sep=1.5pt}]
\begin{scope}
\node (v) at (-0.5, 2) {}; 
\node[draw=none, fill=none, above=2pt] at (v) {$v$};
\node (a) at (-1, 1) {};
\node[draw=none, fill=none, left=2pt] at (a) {$u$};
\node (a1) at (-1.5, 0) {};
\node (a2) at (-0.5, 0) {};
\node (b) at (0, 1) {};
\node (c) at (0, 0) {};
\node[draw=none, fill=none, right=2pt] at (b) {$u'$};
\node[draw=none, fill=none, right=2pt] at (c) {$b$};
\node[draw=none, fill=none, left=2pt] at (a1) {$a$};
\node[draw=none, fill=none, right=0.2pt] at (a2) {$a'$};
\draw (v) -- (a);
\draw (v) -- (b)--(a);
\draw (b) -- (c);
\draw (a) -- (a1);
\draw (a) -- (a2);
\draw (a1) -- (a2);
\end{scope}
\end{tikzpicture}
\caption{$G[\{v,u,u',a, a',b\}]\cong H_{5}$}
\label{fig:H5inCase1}
\end{subfigure}
\caption{Configurations in Case 1 of \Cref{claim:last claim}}
\label{fig:H4H5inCase1inClaim3.9}
\end{figure}
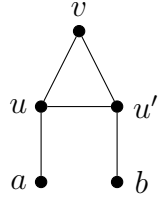
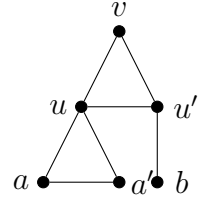
By~\Cref{table:3neg-energy}, we have that $\mathcal{E}_3^{-}(H')\geq 6=|H'|+1$ and $\mathcal{E}_3^{-}(H'')\geq 8=|H''|+2$. Note that $G-H'$ is connected if $|C_i|=1$, and $G-H''$ has at most two connected components if $|C_i|\geq 2$. By~\Cref{thm: additivity}, either $\mathcal{E}_3^{-}(G)\geq\mathcal{E}_3^{-}(H')+\mathcal{E}_3^{-}(G-H')\geq 6+\big( (n-5)-1\big)=n$, or $\mathcal{E}_3^{-}(G)\geq\mathcal{E}_3^{-}(H'')+\mathcal{E}_3^{-}(G-H'')\geq 8+\big( (n-6)-2\big)=n$, each leading to a contradiction.   
    
    \item Assume that there is no vertex of $G-H$ adjacent to any vertex in $C_i$ or $C_j$. Since $uu' \in E(G)$ and $u,u' \in N(v)$, the graph $G$ consists of the triangle $uu'v$ together with three cliques $G[\{u\}\cup C_i]$, $G[\{u'\}\cup C_j]$, and $G-H$, each containing one of $u$, $u'$, and $v$, respectively. These cliques are pairwise disjoint and have no connections between them, except for the edges $uu'$, $uv$, and $u'v$. By~\Cref{lem:ThreeCliques}, $\mathcal{E}_3^{-}(G)\geq n$.
\end{itemize}

\smallskip
\noindent
{\bf Case 2}. \emph{The vertex $u'$ is not adjacent to any vertex of any $C_i$ in $G-N[v]$.}


\smallskip

Since $u'$ is not connected to $C_i$, $H:=G[\{u,u'\}\cup C_i ]$ is not a clique. Note that $H$ is isomorphic to $P_3$ only if $|C_i|=1$, $|S_i|=1$, and $u'$ is an isolated vertex in $G-v$, which is the exceptional case. So we know that $H$ is not isomorphic to $P_3$ as well. Since $u'$ is not adjacent to any $C_j$, we know that $G-H$ is connected. Hence, by~\Cref{claim:H+G-H}, $G-H$ is either a clique or isomorphic to $P_3$. If $G-H$ is isomorphic to $P_3$, then $\Delta=d_G(v)\leq 4$. Since $u$ is adjacent to $u',v$ and all vertices of $S_i$ and $C_i$, we have that $|S_i\setminus\{u\}|+|C_i|+|\{u',v\}|\leq d(u)\leq 4$. Thus $|S_i\cup C_i|\leq 3$. Hence, $|G|=|P_3|+|S_i\cup C_i|+|N[u']|\leq 3+3+3=9$, a contradiction. From now on, we assume that $G-H$ is a clique. 
\begin{figure}[ht]
\centering
\begin{subfigure}[t]{.4\textwidth}
\centering
\begin{tikzpicture}[every node/.style={draw,circle,fill=black,inner sep=1.5pt}]
\begin{scope}[xshift=5cm]
\node (1) at (0,0) {};
\node (2) at (1,0) {};
\node (3) at (0.5,-1) {};
\node (5) at (2,0) {};
\node (4) at (1,1) {};
\draw (1) -- (2);
\draw(1)--(3) -- (2);
\draw (1) -- (4);
\draw (2) -- (4);
\draw (2) -- (5)--(4);
\node[draw=none, fill=none, below=5pt] at (2) {$u$};
\node[draw=none, fill=none, above=2pt] at (4) {$v$};
\node[draw=none, fill=none, below=2pt] at (1) {$u_1$};
\node[draw=none, fill=none, below=2pt] at (3) {$C_i$};
\node[draw=none, fill=none, below=2pt] at (5) {$u'$};
\end{scope}
\end{tikzpicture}
\caption{$G[\{u_1,u,u',v\}\cup C_i]\cong H_{10}$}
\label{fig:H10inCase2ofClaim3.9}
\end{subfigure}
\begin{subfigure}[t]{.5\textwidth}
\centering
\begin{tikzpicture}[every node/.style={draw,circle,fill=black,inner sep=1.5pt}]
\begin{scope}[xshift=0cm]
\node (v) at (0, 2) {};  
\node[draw=none, fill=none, above=2pt] at (v) {$v$};
\node (1) at (-1, 1) {};
\node[draw=none, fill=none, left=2pt] at (1) {$u$};
\node (b) at (0, 1) {};
\node (3) at (-1, 0) {};
\node (c) at (1, 1) {};
\draw (v) -- (1);
\draw (v) -- (b);
\draw (v) -- (c);
\draw (1) -- (b);
\draw (1) -- (3);
\node[draw=none, fill=none, below=2pt] at (b) {$u'$};
\node[draw=none, fill=none, below=2pt] at (c) {$c$};
\end{scope}
\begin{scope}[xshift=4cm]
\node (v) at (0, 2) {}; 
\node[draw=none, fill=none, above=2pt] at (v) {$v$};
\node (a) at (-1, 1) {};
\node[draw=none, fill=none, left=2pt] at (a) {$u$};
\node (a1) at (-1.5, 0) {};
\node (a2) at (-0.5, 0) {};
\node (b) at (0, 1) {};
\node[draw=none, fill=none, below=2pt] at (b) {$u'$};
\node (c) at (1, 1) {};
\node[draw=none, fill=none, below=2pt] at (c) {$c$};
\draw (v) -- (a);
\draw (v) -- (b)--(a);
\draw (v) -- (c);
\draw (a) -- (a1);
\draw (a) -- (a2);
\draw (a1) -- (a2);
\node[draw=none, fill=none, below=2pt] at (0.5, -0.5) {};
\end{scope}
\end{tikzpicture}
\caption{$G[\{u,u',c,v\}\cup C_i]\cong H_{4}~\text{or}~H_{5}$}
\label{fig:H4H5inClaim3.9}
\end{subfigure}
\caption{Configurations in Case 2 of \Cref{claim:last claim}}
\label{fig:LastCase}
\end{figure}

We first claim that $S_i=\{u\}$. Otherwise, assume that $S_i\setminus \{u\}\neq\emptyset$. If there is a vertex $u^*\in S_i\setminus \{u\}$ is adjacent to $u'$, then $u^*$ has the same degree as $v$ in $G-G'$, but $u$ is connected to $C_i$ in $G$, so  $d_G(u^*)>d_G(v)=\Delta$, a contradiction. So none of $ S_i\setminus \{u\}$ is adjacent to $u'$. In this case, $|C_i|=1$, as otherwise, $d_G(u^*)>d_G(v)=\Delta$, again a contradiction. Let $u_1\in S_i\setminus \{u\}$ be an arbitrary vertex. So $u_1$ is not adjacent to $u'$. Now the induced subgraph $H':=G[\{u_1,u,u',v\}\cup C_i]$ is isomorphic to $H_{10}$ as shown in \Cref{fig:H10inCase2ofClaim3.9}. By~\Cref{table:3neg-energy}, we have $\mathcal{E}_3^{-}(H')\approx 7.53>7$. Note that $G-H'$ is a clique. Hence, by~\Cref{thm: additivity}, $
\mathcal{E}_3^{-}(G) \geq \mathcal{E}_3^{-}(H') + \mathcal{E}_3^{-}(G-H') \geq 7+(n-5-1)> n$, a contradiction. 

Thus we have that $S_i=\{u\}$. We then claim that there exists a vertex $c\in N(v)\setminus \{u\}$ that is not adjacent to $u'$. Otherwise, since $u'$ is adjacent to all vertices of $N(v)\setminus \{u\}$, recalling that $G-G[\{u,u'\}\cup C_i]$ is a clique, $G-u$ is a union of two distinct cliques. By~\Cref{lem:UnionOfCliques},  $
\mathcal{E}_3^{-}(G)\geq n$, a contradiction. We consider two subcases. 

\begin{itemize}
    \item If $|C_i|\leq 2$, then $G_1:=G[\{u,u',c,v\}\cup C_i]$ is isomorphic to either $H_4$ (when $|C_i|=1$) or $H_5$ (when $|C_i|=2$), as shown in~\Cref{fig:H4H5inClaim3.9}. By~\Cref{table:3neg-energy},  $\mathcal{E}_3^{-}(H_4)\approx 6.447$ and $\mathcal{E}_3^{-}(H_5)\approx 8.026$; and so $\mathcal{E}_3^{-}(G_1)\geq |G_1|+1$. Since $G-G_1$ is obtained from $G-H$ by deleting two vertices $c$ and $v$, noting that $G-H$ is a clique, $G-G_1$ is connected. Hence, by~\Cref{thm: additivity}, $\mathcal{E}_3^{-}(G) \geq \mathcal{E}_3^{-}(G_1) + \mathcal{E}_3^{-}(G-G_1) \geq (|G_1|+1)+(n-|G_1|-1) = n$, a contradiction. 

    \item If $|C_i|\geq 3$, then $G_2:=G[\{u,u',v\}\cup C_i]$ is a graph consisting of a triangle $abv$ together with a clique of size at least $4$ containing $a$ such that $\{v,u'\}$ is not connected to any other vertex of this clique. Thus $|G_2|\geq 6$. By~\Cref{lem:K_(n-2)+K_2}, $\mathcal{E}_3^{-}(G_2)\geq |G_2|+1$. Since $G-H$ is a clique, $G-G_2$ is also a clique. Hence, by~\Cref{thm: additivity}, $\mathcal{E}_3^{-}(G)\geq\mathcal{E}_3^{-}(G_2)+\mathcal{E}_3^{-}(G-G_2)\geq (|G_2|+1)+(n-|G_2|-1)=n$, a contradiction.
\end{itemize}
This completes the proof of the claim.
\end{proof*}

\medskip
What remains is to show that  all these structures are ``vertex-disjoint". 

By~\Cref{claim:NonAdjacentP3}, $u\in N(v)$ in a $P_3$ of any type $i$ (for $i\in [3]$) is not adjacent to another $u'\in N(v)$; and thus such a vertex $u$ cannot be in any subgraph of the form $G[S_i\cup C_i]$ or cliques formed by vertices not adjacent to any of $C_i$'s. There are two other kinds of $P_3$'s: $G[\{u,u',a\}]$ and $G[\{a,a', a''\}]$ where $u,u'\in N(v)$ and $a', a', a''$ are vertices not adjacent to any vertex of any $C_i$. See~\Cref{fig:G-v}. 

For $G[\{u,u',a\}]$ isomorphic to $P_3$ with $ua\in E(G)$, we apply~\Cref{claim:last claim} to $K=G[\{S_{i_a}\}\cup \{C_{i_a}\}]$ with $S_{i_a}$ containing $u$ and $C_{i_a}$ containing $a$. Then we know that $|S_{i_a}|=1$, $|C_{i_a}|=1$, and $u'$ is an isolated vertex in $G-v$. Hence, it follows from \Cref{claim:last claim} that neither $u$ nor $u'$ is adjacent to any other $u''\in N(v)$; thus, neither $u$ nor $u'$ is in any subgraph of the form $G[S_i\cup C_i]$ or cliques formed by vertices not adjacent to any of $C_i$'s.

For $G[\{a,a', a''\}]$ isomorphic to $P_3$ with $aa', a'a''\in E(G)$, recall that $a', a', a''$ are vertices not adjacent to any vertex of any $C_i$. Then none of $a,a',$ and $a''$ is in any clique of the form $G[S_i\cup C_i]$. Moreover, by~\Cref{claim:non-S_i-vertices}, they are not in any clique formed by vertices not adjacent to any of $C_i$'s.

Therefore, each clique in $G-v$ is vertex disjoint from any kind of induced $P_3$ in $G-v$.

By~\Cref{claim:non-S_i-vertices}, cliques in $N(v)$ are vertex disjoint; and by~\Cref{claim:last claim}, cliques of the form $G[S_i\cup C_i]$ are pairwise vertex disjoint. If a vertex $u\in N(v)$ is contained in both of $G[S_i\cup C_i]$ for some $i$ and a clique $X$ (with $|X|\geq 2$) induced only by vertices in $N(v)$, then $u$ is adjacent to another vertex $u'\in V(X)\setminus \{u\}$, a contradiction to~\Cref{claim:last claim}. 

We now conclude that $G-v$ is a disjoint union of some $P_3$'s and some cliques. By~\Cref{lem:P3+cliques}, $\mathcal{E}_3^{-}(G) \geq n$, a contradiction. We complete the proof of the main theorem.  
\end{proof}

\section*{Acknowledgments}
Z. Wang is partially supported by National Natural Science Foundation of China (No. 12301444) and the Fundamental Research Funds for the Central Universities, Nankai University.
X.-D. Zhang is partially supported by  the National Natural Science Foundation of China (Nos. 12371354, 12161141003) and Science and Technology Commission of Shanghai Municipality (No. 22JC1403600), and the Fundamental Research Funds for the Central Universities. 

\bibliographystyle{plain}  
\bibliography{reference}

\newpage
\section*{Appendix}

\noindent
{\bf \Cref*{lem:K_{n-1}+v}.}
\emph{Let $n$ be a positive integer with $n\geq 4$. If $G$ is a graph formed from $K_{n-1}$ by attaching a pendant vertex to one of its vertices, then $\mathcal{E}_3^{-}(G)\geq n$.}

\begin{proof}
Assume that $n\geq 4$. 
Let $V(G)=\{v_1,v_2,\ldots,v_{n-1},u\},$
where $G[\{v_1,\ldots,v_{n-1}\}]$ is isomorphic to $K_{n-1}$ and $u$ is a pendant vertex adjacent only to $v_1$.
We consider the vertex partition $\mathcal{P}=\{ \{u\},\{v_1\},\{v_2,\ldots,v_{n-1}\} \}.$
Let $M_{\mathcal{P}}$ denote the corresponding quotient matrix of $G/\mathcal{P}$. We have that 
$$
M_{\mathcal{P}}=
\begin{pmatrix}
0 & 1 & 0\\
1 & 0 & n-2\\
0 & 1 & n-3
\end{pmatrix}.
$$ A direct computation yields
$$f_n(\lambda):=\det(\lambda I-M_\mathcal{P})
=\lambda^3-(n-3)\lambda^2-(n-1)\lambda+(n-3).$$
This cubic polynomial has exactly one negative root $\lambda_0$. 

Observe that the induced subgraph $G[\{v_2,\ldots,v_{n-1}\}]$ is isomorphic to $K_{n-2}$. Hence, $-1$ is an eigenvalue of the adjacency matrix of $G$ with multiplicity $n-3$.
Therefore, the negative eigenvalues of $G$ are $-1$ (with multiplicity $n-3$) and $\lambda$.
Consequently, $$\mathcal{E}_3^{-}(G)=n-3+|\lambda|^3.$$
It suffices to show that $|\lambda_0|^3\geq 3$ for $n\geq 4$.

On the one hand, $f_n(-\sqrt[3]{3})= (\sqrt[3]{3} - \sqrt[3]{9} + 1)n + 3\sqrt[3]{9} - \sqrt[3]{3} - 6,$ where $\sqrt[3]{3} - \sqrt[3]{9} + 1\approx 0.362>0$. So $f_n(-\sqrt[3]{3})$ increases as $n$ increases. In particular, we note that $f_4(-\sqrt[3]{3})=3\sqrt[3]{3} - \sqrt[3]{9} - 2 \approx 0.246 > 0$. So $f_n(-\sqrt[3]{3})>0$ for every $n\geq 4$.

On the other hand, for a fixed $n\geq 4$, $\lim_{\lambda \to -\infty} f_n(\lambda) = -\infty.$ By continuity, the negative root $\lambda_0$ of $f(\lambda)$ satisfies
$\lambda_0 < -\sqrt[3]{3},$
which implies $|\lambda_0| > \sqrt[3]{3}$. Therefore, $|\lambda_0|^3 > 3$. This completes the proof.
\end{proof}

\noindent
{\bf \Cref*{lem:K_(n-2)+K_2}.}
\emph{Let $n$ be a positive integer with $n\geq 6$. If $G$ is a graph formed from $K_{n-2}$ by adding an extra edge $uv$ and connecting $u$ and $v$ to one same vertex of $K_{n-2}$, then $\mathcal{E}_3^{-}(G)>n+1$.}

\begin{proof}
Let $y$ be the vertex of $K_{n-2}$ adjacent to both $u$ and $v$.
We consider the vertex partition $
\mathcal{P}=\{\{u\},\{v\},\{y\},V(K_{n-2})\setminus\{y\}\}.$
This partition is equitable, and let $M_{\mathcal{P}}$ denote the corresponding quotient matrix. We have that  
$$M_{\mathcal{P}}=
\begin{pmatrix}
0 & 1 & 1 & 0\\
1 & 0 & 1 & 0\\
1 & 1 & 0 & n-3\\
0 & 0 & 1 & n-4
\end{pmatrix}.$$ 
A direct computation gives
$$f_n(\lambda):=\det(\lambda I-M_\mathcal{P})
=\lambda^4-(n-4)\lambda^3-n\lambda^2+(3n-14)\lambda+(3n-11).$$
One may check that $f_n(-1)=0$ and the derivative $f'_n(-1)=2n-6>0$ for $n\geq 6$, so $-1$ is a simple root of $f(\lambda)$. Since $G[V(G)\setminus\{y\}]\cong K_{n-3}\cup K_2$, the eigenvalue $-1$ appears in $\operatorname{Spec}(G)$ with multiplicity $n-3$. Thus the negative eigenvalues of $G$ consist of $-1$ (with multiplicity $n-3$) and one additional eigenvalue $\lambda_0$ which is the negative root of $f(\lambda)$. Hence,
$$\mathcal{E}_3^{-}(G)=n-3+|\lambda_0|^3.$$
It suffices to show that $|\lambda_0|^3>4$ for $n\geq 6$.

On the one hand, $f_n(-\sqrt[3]{4})=(7-\sqrt[3]{4^2}-3\sqrt[3]{4})n+18\sqrt[3]{4}-27$, where $7-\sqrt[3]{4^2}-3\sqrt[3]{4}<0$. So the function $f_n(-\sqrt[3]{4})$ (in terms of $n$) decreases as $n$ increases. In particular, we note that $f_6(-\sqrt[3]{4})=15-6\sqrt[3]{16}\approx -0.118  <0$. So $f_n(-\sqrt[3]{4})<0$ for every $n\geq 6$.

On the other hand, for a fixed $n\geq 6$, $\lim_{\lambda \to -\infty} f_n(\lambda) = +\infty.$ By continuity, the negative root $\lambda_0$ of $f(\lambda)$ satisfies $\lambda_0 < -\sqrt[3]{4},$ which implies $|\lambda_0| > \sqrt[3]{4}$. Therefore, $|\lambda_0|^3 > 4$. This completes the proof.
\end{proof}

\noindent
{\bf \Cref*{lem:star-plus}.}
\emph{Let $n$ be a positive integer. If $G$ is a graph formed from $K_{1,n}$ by subdividing $t$ edges each once, then $\operatorname{Spec}(G)=\{\pm\sqrt{x_1},\pm\sqrt{x_2},0^{n-t-1},1^{t-1},(-1)^{t-1}\}$, where $x_1$ and $x_2$ are the roots of $x^2-(n+1)x+(n-t)=0$ with $x_1>x_2>0$.}

\begin{proof}
Let $v$ be the center of the original $K_{1,n}$.
Let $P$ be the set of $t$ leaves whose incident edges are subdivided,
$Q$ be the set of the remaining $n-t$ leaves, and
$R$ be the set of those subdivided vertices.
We consider the vertex partition
$\mathcal{P}=\{\{v\},P,R,Q\}.$
This partition is equitable, and let $M_{\mathcal{P}}$ denote the corresponding quotient matrix. We have that 
$$
A(G)=
\begin{pmatrix}
0 & 1 &\cdots & 1 & 0 & \cdots & 0 & 1 &\cdots & 1\\
1 & 0 &\cdots & 0 & 1 & \cdots & 0 & 0 &\cdots & 0\\
\vdots & \vdots & \ddots & \vdots & \vdots & \ddots & \vdots 
& \vdots & \ddots & \vdots \\
1 & 0 &\cdots & 0 & 0 & \cdots & 1 & 0 &\cdots &0\\
0 & 1 &\cdots & 0 & 0 & \cdots & 0 & 0 &\cdots &0\\
\vdots &\vdots & \ddots & \vdots & \vdots & \ddots & \vdots & \vdots & \ddots & \vdots \\
0 & 0 & \cdots & 1 & 0 &\cdots & 0 & 0 &\cdots &0\\
1 & 0 &\cdots & 0 & 0 & \cdots & 0 & 0 &\cdots &0\\
\vdots & \vdots & \ddots & \vdots & \vdots & \ddots & \vdots & \vdots & \ddots & \vdots \\
1 & 0 &\cdots & 0 & 0 & \cdots & 0& 0 &\cdots & 0\\
\end{pmatrix}
\text{~~and~~} M_{\mathcal{P}}=
\begin{pmatrix}
0 & t & 0 & n-t\\
1 & 0 & 1 & 0\\
0 & 1 & 0 & 0\\
1 & 0 & 0 & 0
\end{pmatrix}.$$
A direct computation yields
$$\det(\lambda I-M_{\mathcal{P}})
=\lambda^4-(n+1)\lambda^2+(n-t).$$
This polynomial has four roots, denoted by
$\pm \sqrt{x_1}$ and $\pm \sqrt{x_2},$ where $x_1$ and $x_2$ are the two roots of 
$x^2-(n+1)x+(n-t)=0$ with $x_1>x_2>0$.

Moreover, for the adjacency matrix $A(G)$, we may observe that $$(0,\ldots,0,1,-1,0,\ldots,0), (0,\ldots,0,1,0,-1,\ldots,0),      \ldots, \text{~and~} (0,\ldots,0,1,0,0,\ldots,-1))$$ (with the value $1$ fixed in the $2t+2$-th coordinate) are the eigenvectors of the eigenvalue $0$ (with multiplicity $n-t-1$), 
$$(0,1,-1,0,\ldots,0,1,-1,0,\ldots,0,0,\ldots,0),(0,1,0,-1,\ldots,0,1,0,-1,\ldots,0,0,\ldots,0),\ldots, $$ $$\text{~and~}(0,1,0,\ldots,-1,1,0,\ldots,-1,0,\ldots,0)$$ (with two $1$'s fixed in the second and the $t+2$-th coordinates) are the eigenvectors of the eigenvalue $1$ (with multiplicity $t-1$), and 
$$(0,1,-1,0,\ldots,0,-1,1,0,\ldots,0,0,\ldots,0),(0,1,0,-1,\ldots,0,-1,0,1,\ldots,0,0,\ldots,0),\ldots, $$ $$\text{~and~}(0,1,0,\ldots,-1,-1,0,\ldots,1,0,\ldots,0)$$ (with the first $1$ fixed in the second coordinate and the second $-1$ fixed in the $2t+2$-th coordinate) are the eigenvectors of the eigenvalue $-1$ (with multiplicity $t-1$). 

Therefore, $\operatorname{Spec}(G)=\{\pm\sqrt{x_1},\pm\sqrt{x_2},0^{\,n-t-1},1^{\,t-1},(-1)^{\,t-1}\}.$
\end{proof}
\end{document}